\documentclass[reqno]{amsart}
\usepackage[left=1in,right=1in,top=1in,bottom=1in]{geometry}
 \usepackage{microtype}
\usepackage{hyperref}
         
\usepackage{amsmath}             
\usepackage{amsfonts}             
\usepackage{amsthm}               
\usepackage{amsbsy}
\usepackage{amssymb}
\usepackage{amsthm}
\usepackage[nobysame]{amsrefs}

\newtheorem{thm}{Theorem}[section]
\newtheorem{lem}[thm]{Lemma}
\newtheorem{prop}[thm]{Proposition}

\theoremstyle{definition}
\newtheorem{defn}[thm]{Definition}

\newtheorem{rem}[thm]{Remark}

\newcommand{\bC}{{\mathbb{C}}}
\newcommand{\bN}{{\mathbb{N}}}

\newcommand{\A}{{\mathcal{A}}}

\newcommand{\I}{{\mathcal{I}}}

\newcommand{\M}{{\mathcal{M}}}

\newcommand{\X}{{\mathcal{X}}}

\newcommand{\qqand}{\qquad\text{and}\qquad}

\newcommand{\inv}{{\langle -1 \rangle}}
\newcommand{\op}{{\mathrm{op}}}

\usepackage{tikz}
\usetikzlibrary{shapes,snakes,calc,arrows}
\usetikzlibrary{decorations.pathreplacing,shapes.geometric}
\usetikzlibrary{calc,positioning}
\tikzset{Box/.style={very thick, rounded corners}}
\tikzset{marked/.style={star, star point height = .75mm, star points =5, fill=black,minimum size=2mm, inner sep=0mm} }
\tikzset{verythickline/.style = {line width=7pt}}
\tikzset{thickline/.style = {line width=5pt}}
\tikzset{medthick/.style = {line width=3pt}}
\tikzset{med/.style = {line width=2pt}}
\tikzset{count/.style = {fill=white,circle,draw,thin, inner sep=2pt}}
\tikzset{rcount/.style = {fill=white,rectangle,draw,thin,inner sep=2pt, rounded corners}}
\tikzset{cpr/.style = {draw,fill=white,rectangle,thin, rounded corners}}

\definecolor{ggreen}{HTML}{00BB33}

\begin{document}

\nocite{*}

\title[A Combinatorial Approach to the Opposite Bi-Free Partial $S$-Transform]{A Combinatorial Approach to the \\ Opposite Bi-Free Partial $S$-Transform}

\author{Paul Skoufranis}
\address{Department of Mathematics and Statistics, York University, 4700 Keele Street, Toronto, Ontario, M3J 1P3, Canada}
\email{pskoufra@yorku.ca}

\subjclass[2010]{46L54, 46L53}
\date{\today}
\keywords{Bi-Free Probability, Partial Bi-Free $S$-Transform, Bi-Free Convolutions}

\begin{abstract}
In this paper, we present a combinatorial approach to the opposite 2-variable bi-free partial $S$-transforms where the opposite multiplication is used on the right.  In addition, extensions of this partial $S$-transforms to the conditional bi-free and operator-valued bi-free settings are discussed.
\end{abstract}

\maketitle

\section{Introduction}

Since \cites{V1986, V1987} the use of analytic functions as transformations has been a vital part of free probability.  With the inception of bi-free probability by Voiculescu in \cite{V2014} it has been natural to ask whether extensions of past results hold and can be used to develop new theories.  Using analytic techniques similar to those from \cite{H1997}, Voiculescu in \cites{V2016-1, V2016-2} developed  bi-free partial $R$- and $S$-transforms to study additive and multiplicative bi-free convolution of bi-partite systems.  Furthermore, due to the combinatorial breakthroughs in \cites{CNS2015-1, CNS2015-2}, the author of this paper was able to develop new proofs to these results in \cites{S2016-1, S2016-2} using the combinatorics of cumulants.  These alternative proofs were essential in extending the bi-free partial $R$-transform to the conditional bi-free setting in \cites{GS2016-1, GS2016-2} and extending both transformations to the operator-valued bi-free setting in \cite{S2016-3}.  The bi-free partial transforms developed in these papers were instrumental in developing bi-free infinite divisibility theorems such as those in \cites{GHM2015, HW2016, GS2016-1}.

However, one important question remained with respect to the bi-free partial $S$-transform.  Given a non-commutative probability space $(\A, \varphi)$ and elements $a, b \in \A$ with $\varphi(a), \varphi(b) \neq 0$, in \cites{V2016-2, S2016-1} a holomorphic function $S_{a,b}(z,w)$ defined on a neighbourhood of $(0,0)$ was constructed.  The important property of $S_{a,b}(z,w)$ was that if $(a_1, b_1)$ and $(a_2, b_2)$ are pairs in $\A$ with non-zero moments that are bi-free with respect to $\varphi$, then
\[
S_{a_1a_2, b_1b_2}(z, w) = S_{a_1, b_1}(z,w) S_{a_2, b_2}(z,w);
\]
that is, the product of the bi-free partial $S$-transforms of $(a_1, b_1)$ and $(a_2, b_2)$ is the bi-free partial $S$-transform of the natural product of $(a_1, b_1)$ and $(a_2, b_2)$.  As it was demonstrated in \cite{CNS2015-2}*{Theorem 5.2.1} that the use of the opposite multiplication on the right-hand component of the pairs leads to combinatorial structures similar to those in free probability (see Subsection \ref{subsec:Bi-Krew}), it was natural to ask whether there was an opposite bi-free partial $S$-transform; that is a holomorphic function $S^\op_{a,b}(z,w)$ defined on a neighbourhood of $(0,0)$ such that if $(a_1, b_1)$ and $(a_2, b_2)$ are pairs in $\A$ with non-zero moments that are bi-free with respect to $\varphi$, then
\begin{align}
S^\op_{a_1a_2, b_2b_1}(z, w) = S^\op_{a_1, b_1}(z,w) S^\op_{a_2, b_2}(z,w). \label{eq:S-op-formula-intro}
\end{align}

At the time of \cite{S2016-1}, the author endeavoured to solve this problem.  Using similar combinatorial techniques to those in \cite{S2016-1} that are slightly more intricate due to a recursive nature, Lemma \ref{lem:K-expression} of this paper was obtainable.  The remaining gap to obtaining an $S^\op_{a,b}(z,w)$ with the desired properties was rearranging the expression in Lemma \ref{lem:K-expression} in the appropriate manner and this remained elusive.

In \cite{HW2017}, a deeper study of bi-free partial $S$-transforms was performed from an analytical standpoint.  In particular, in order to develop a theory of bi-free multiplicative infinite divisibility, a formula for the opposite bi-free partial $S$-transform was required and discovered.  Consequently, due to the expression in \cite{HW2017}, it is now possible to complete the combinatorial approach started in \cite{S2016-1}.

This paper contains four sections excluding this introduction, which are structured as follow.  In Section \ref{sec:Transforms} we recall the necessary transformations from free and bi-free probability necessary in this paper.  In particular Proposition \ref{prop:S-op-new-defn} demonstrates an alternative expression for $S^\op_{a,b}(z,w)$ than that which was given in \cite{HW2017}.  The paper continues by recalling the necessary combinatorial tools in Section \ref{sec:Tools} and proving equation \eqref{eq:S-op-formula-intro} in Section \ref{sec:Proof}.

As the main uses of the combinatorial proofs of the bi-free partial transforms thus far has been the extension of these transformations to the operator-valued bi-free and conditional bi-free settings, we conclude this paper with Section \ref{sec:OtherSettings} analyzing whether the results of this paper extend to these settings.  Unfortunately, the answers appear to be negative.

\section{The Transforms}
\label{sec:Transforms}

In this section, we recall several transformations in free and bi-free probability and set the notation for the paper, which is consistent with the notation in  \cites{V2016-2, S2016-1}.  

\subsection{Free Transforms}

Let $(\A, \varphi)$ be a non-commutative probability space (that is, a unital algebra $\A$ with a linear functional $\varphi : \A \to \bC$ such that $\varphi(I) = 1$).  We will always assume that $\A$ is a Banach algebra with norm $\left\| \, \cdot \, \right\|_\A$ and $\varphi$ is a continuous linear function.  These assumptions are made so that the power series defined in this paper converge absolutely and define holomorphic functions.

Fix $a \in \A$.  Recall the \emph{moment series of $a$} is defined to be
\begin{align*}
h_a(z) := \varphi((I-az)^{-1}) = \sum_{n\geq 0} \varphi(a^n) z^{n} 
\end{align*}
and, if $\kappa_n(a)$ denotes the $n^{\mathrm{th}}$ free cumulant of $a$, the \emph{cumulant series of $a$} is
\begin{align*}
c_a(z) := \sum_{n\geq 1} \kappa_n(a) z^n.
\end{align*}

\begin{rem}
\label{rem:convergence-free}
We momentarily pause to discuss the convergence of the above series.  Notice for all $n \in \bN$ that
\[
\left|\varphi(a^n) z^n\right| \leq \left\|\varphi\right\| \left\|a\right\|^n_\A |z|^n.
\]
Hence, provided $|z| < \left\|a\right\|_\A^{-1}$, we see that $h_a(z)$ converges absolutely.

The moment-cumulant formula will immediately imply that $c_a(z)$ converges absolutely if $z$ is sufficiently small.  To see this, let $NC(n)$ denote the set of non-crossing partitions on $n$ elements, let $1_n$ denote the full partition $1_n = \{\{1, \ldots, n\}\}$, and let $\mu_{NC}$ denote the M\"{o}bius function on the set of non-crossing partitions.  Then for each $n \in \bN$, $\kappa_n(a)$ is a sum of at most $|NC(n)|$ 
terms of the form $\mu_{NC}(\pi, 1_n)$ for some $\pi\in NC(n)$ times a product of moments of $a$ each of which is  bounded in absolute value by $\max\{ 1, \left\|\varphi\right\|^n \} \left\|a \right\|^n_\A$.  Due to the lattice  structure, $|\mu_{NC}(\pi, 1_n)|$ is at most $|NC(n)|$.  Since $|NC(n)|$ is the $n^\mathrm{th}$ Catalan number $C_n$ and since $C_n \leq 4^n$, $c_a(z)$ converges absolutely if $z$ is sufficiently small.
\end{rem}

To define the $S$-transform of $a$, we illustrate two possible methods.  For both we must assume $\varphi(a) \neq 0$. 
For the first, let $\psi_a(z) := h_a(z) - 1$.  Since $\psi_a(0) = 0$ and $\psi'_a(z) = \varphi(a) \neq 0$, $\psi_a(z)$ has an inverse under composition, denoted $\psi^{\inv}_a(z)$ (see \cite{D2006} for analytic arguments).  We define $\X_a(z) := \psi^{\inv}_a(z)$ so that
\begin{align}
h_a(\X_a(z)) = 1 + \psi_a(\X_a(z)) = 1+z. \label{eq:X-into-moment}
\end{align}
The $S$-transform of $a$ can then defined by \cite{V1987} to be
\begin{align}
S_a(z) = \frac{1+z}{z} \X_a(z).  \label{eq:S-X-form}
\end{align}
For the second, notice since $c_a(0) = 0$ and $c'_a(z) = \kappa_1(a) = \varphi(a) \neq 0$, $c_a(x)$ has an inverse under composition, denoted $c_a^{\inv}(z)$.  The $S$-transformation of $a$ can also be defined by \cite{NS1997} to be
\begin{align}
S_a(z) = \frac{1}{z} c_a^{\inv}(z).  \label{eq:S-cumulants-form}
\end{align}

\subsection{Bi-Freeness}

For more background on scalar-valued bi-free probability, we refer the reader to the summary in \cite{CNS2015-2}*{Section 2}.  

For a map $\chi : \{1,\ldots, n\} \to \{\ell, r\}$, the set of bi-non-crossing partitions on $\{1,\ldots, n\}$ associated to $\chi$ is denoted by $BNC(\chi)$ and $1_\chi$ is used to denote the full partition. Given two elements $\pi, \sigma \in BNC(\chi)$, we will use $\pi \vee \sigma$ to denote the smallest element of $BNC(\chi)$ that is larger than both $\pi$ and $\sigma$ under the usual ordering of refinement.

Given a non-commutative probability space $(\A, \varphi)$ and elements $\{a_n\}^n_{n=1} \subseteq \A$, the $(\ell, r)$-cumulant associated to a map $\chi : \{1,\ldots, n\} \to \{\ell, r\}$ is denoted $\kappa_{\chi}(a_1, \ldots, a_n)$.  Given a $\pi \in BNC(\chi)$, each block $V$ of $\pi$ corresponds to the bi-non-crossing partition $1_{\chi_V}$ for some $\chi_V : V \to \{\ell, r\}$ (where the ordering on $V$ is induced from $\{1,\ldots, n\}$).  We denote
\[
\kappa_\pi(a_1, \ldots, a_n) = \prod_{V \text{ a block of }\pi} \kappa_{1_{\chi_V}}((a_1,\ldots, a_n)|_V)
\]
where $(a_1,\ldots, a_n)|_V$ denotes the $|V|$-tuple where indices not in $V$ are removed.  

For $n,m\geq 0$, we will only ever consider the maps $\chi_{n,m} : \{1,\ldots, n+m\} \to \{\ell, r\}$ such that $\chi(k) = \ell$ if $k \leq n$ and $\chi(k) = r$ if $k > n$.  For notation purposes, it will be useful to think of $\chi_{n,m}$ as a map on $\{1_\ell, 2_\ell, \ldots, n_\ell, 1_r, 2_r, \ldots, m_r\}$ under the identification $k \mapsto k_\ell$ if $k \leq n$ and $k \mapsto (k-n)_r$ if $k > n$. Furthermore, we denote $BNC(n,m)$ for $BNC(\chi_{n,m})$, $1_{n,m}$ for $1_{\chi_{n,m}}$, and, for $n,m\geq 1$, $\kappa_{n,m}(a_1, \ldots, a_n, b_1, \ldots, b_m)$ for $\kappa_{1_{n,m}}(a_1, \ldots, a_n, b_1, \ldots, b_m)$.  In the case that $a_1 = a_2 = \cdots = a_n = a$ and $b_1 = b_2 = \cdots = b_m = b$, we will use $\kappa_\pi(a,b)$ to denote $\kappa_{\pi}(a_1, \ldots, a_n, b_1, \ldots, b_m)$ for all $\pi \in BNC(n,m)$.  Consequently, for $n,m \geq 1$, we have $\kappa_{n,m}(a,b) = \kappa_{1_{n,m}}(a,b)$, $\kappa_{n,0}(a,b) = \kappa_n(a)$, and $\kappa_{0,m}(a,b) = \kappa_{n}(b)$.

\subsection{Bi-Free Series}

Given a Banach non-commutative probability space $(\A, \varphi)$ and a pair $(a, b)$ of elements in $\A$, we define the \emph{ordered joint moment and cumulant series of the pair $(a,b)$} to be
\[
H_{a,b}(z,w) := \sum_{n,m\geq 0} \varphi(a^nb^m) z^n w^m \qqand C_{a,b}(z,w) := \sum_{n,m\geq 0} \kappa_{n,m}(a,b) z^n w^m
\]
respectively (where $\kappa_{0,0}(a,b) = 1$).  Note $H_{a,b}(z,w)$ and $C_{a,b}(z,w)$ are not sufficient to describe all joint moments of $a$ and $b$ unless $a$ and $b$ commute.

\begin{rem}
\label{rem:convergence-bi-free}
We momentarily pause again to discuss the convergence of the above series.  Notice for all $n \in \bN$ that
\[
\left|\varphi(a^nb^m) z^nw^m\right| \leq \left\|\varphi\right\| \left\|a\right\|^n_\A\left\|b\right\|^m_\A |z|^n |w|^m.
\]
Hence, provided $|z| < \left\|a\right\|_\A^{-1}$ and $|w| < \left\|b\right\|_\A^{-1}$, we see that $H_{a,b}(z)$ converges absolutely.  As the combinatorics of bi-non-crossing partitions is identical to that of non-crossing partitions up to a permutation, similar arguments to those in Remark \ref{rem:convergence-free} implies that $C_{a,b}(z,w)$ converges if $z$ and $w$ are sufficiently small.
\end{rem}

Note \cite{S2016-2}*{Theorem 7.2.4} demonstrates that
\begin{align}
h_a(z) + h_b(w) = \frac{h_a(z)h_b(w)}{H_{a,b}(z,w)}  + C_{a,b}(z h_a(z), w h_b(w)) \label{eq:bi-moment-and-bi-cumulant}
\end{align}
through combinatorial techniques. In addition, for notational purposes it will be helpful to consider the series
\begin{align}
K_{a,b}(z,w) = \sum_{n,m\geq 1} \kappa_{n,m}(a,b) z^n w^m = C_{a,b}(z,w) - c_a(z) - c_b(w) - 1. \label{eq:bi-K-to-C}
\end{align}

\subsection{Bi-Free Partial $S$-Transforms}

Finally, with the above notation, we can define the bi-free partial $S$-transform and the opposite bi-free partial $S$-transform.  

\begin{defn}[\cite{V2016-2}*{Definition 2.1}, \cite{S2016-1}*{Proposition 4.2}]
Let $(a, b)$ be a two-faced pair in a non-commutative probability space $(\A, \varphi)$ with $\varphi(a) \neq 0$ and $\varphi(b) \neq 0$.  The \emph{$2$-variable partial bi-free $S$-transform of $(a, b)$} is the holomorphic function on $(\bC \setminus \{0\})^2$ near $(0,0)$ defined by
\begin{align*}
S_{a,b}(z, w) &= \frac{z+1}{z}\frac{w+1}{w} \left( 1-\frac{1+z+w}{H_{a,b}(\X_a(z), \X_b(w))}\right)\\
&= 1 + \frac{1+z+w}{zw} K_{a,b}\left( c^{\inv}_a(z),  c^{\inv}_b(w)\right)
\end{align*}
\end{defn}

\begin{defn}[\cite{HW2017}*{Definition 2.7}]
Let $(a, b)$ be a two-faced pair in a Banach non-commutative probability space $(\A, \varphi)$ with $\varphi(a) \neq 0$ and $\varphi(b) \neq 0$.  The opposite $2$-variable partial bi-free $S$-transform of $(a, b)$ is the holomorphic function on $(\bC \setminus \{0\})^2$ near $(0,0)$ defined by
\begin{align}
S^\op_{a,b}(z, w) = \frac{w(z+1)}{z(w+1)}\frac{H_{(a,b)}(\X_a(z), \X_b(w)) - (w+1)}{H_{(a,b)}(\X_a(z), \X_b(w)) - (z + 1)}. \label{eq:S-op}
\end{align}
\end{defn}

Note equation \eqref{eq:S-op} written slightly differently than in \cite{HW2017}*{Definition 2.7} but is the same formula.  However, we desire an alternative formula along the lines of \cite{S2016-1}*{Proposition 4.2}.  To simplify discussions, we will demonstrate the equality in the case $\varphi(a) = \varphi(b) = 1$. See Remark  \ref{rem:assume-unital} for why we can do this.

\begin{prop}
\label{prop:S-op-new-defn}
If $(a,b)$ is a two-faced pair in a Banach non-commutative probability space $(\A, \varphi)$ with $\varphi(a) = \varphi(b) =1$, then as holomorphic functions on $(\bC \setminus \{0\})^2$ near $(0,0)$,
\begin{align}
S^\op_{a,b}(z,w) = \frac{1 + \frac{1}{z} K_{a,b}\left( c^{\inv}_a(z),  c^{\inv}_b(w)\right)}{1 + \frac{1}{w} K_{a,b}\left( c^{\inv}_a(z),  c^{\inv}_b(w)\right)} =  \frac{1 + \frac{1}{z} K_{a,b}\left( zS_a(z),  wS_b(w)\right)}{1 + \frac{1}{w} K_{a,b}\left( zS(z),  wS_b(w)\right)}. \label{eq:my-S-op}
\end{align}
\end{prop}
\begin{proof}
First note that the two expression for $S^\op_{a,b}(z,w)$ in equation \eqref{eq:my-S-op} are equal by equation \eqref{eq:S-cumulants-form}.  Furthermore, these expressions are well defined.  Indeed, as $\varphi(a) = \varphi(b) = 1$, if $(z, w)$ are sufficiently near $(0, 0)$ then $c^{\inv}_a(z) = z + O(z^2)$ and $c^{\inv}_b(w) = w + O(w^2)$ and thus $\frac{1}{z} K_{a,b}\left( c^{\inv}_a(z),  c^{\inv}_b(w)\right)$ and $\frac{1}{w} K_{a,b}\left( c^{\inv}_a(z),  c^{\inv}_b(w)\right)$ are well-defined and $1 + \frac{1}{w} K_{a,b}\left( c^{\inv}_a(z),  c^{\inv}_b(w)\right) \neq 0$ for sufficiently small $(z,w)$.  

When $(z,w)$ is sufficiently near $(0, 0$), $(\X_a(z), \X_b(w))$ is near $(0, 0)$ and thus $H_{(a,b)}(\X_a(z), \X_b(w))$ is near 1 and invertible.  Hence we obtain from equation \eqref{eq:S-op} that
\[
S^\op_{a,b}(z, w) = \frac{w(z+1)}{z(w+1)}\left(\frac{1  - \frac{w+1}{H_{(a,b)}(\X_a(z), \X_b(w))}}{1  - \frac{z + 1}{H_{(a,b)}(\X_a(z), \X_b(w))}}\right).
\]
Using equations (\ref{eq:X-into-moment}, \ref{eq:S-X-form}, \ref{eq:S-cumulants-form}, \ref{eq:bi-moment-and-bi-cumulant}), we obtain that
\begin{align*}
\frac{1}{H_{a,b}\left(\X_a(z), \X_b(w)\right)} = \frac{1}{1+z} + \frac{1}{1+w} -\frac{1}{1+z}\frac{1}{1+w}C_{a,b}\left(c^{\inv}_a(z), c^{\inv}_b(w)\right).
\end{align*}
Therefore, using equations \eqref{eq:bi-K-to-C}, we obtain that
\begin{align*}
S^\op_{a,b}(z, w) &= \frac{w(z+1)}{z(w+1)}\left(\frac{1 - (w+1)\left[\frac{1}{1+z} + \frac{1}{1+w} -\frac{1}{1+z}\frac{1}{1+w}C_{a,b}\left(c^{\inv}_a(z), c^{\inv}_b(w)\right)\right]}{1 - (z+1)\left[\frac{1}{1+z} + \frac{1}{1+w} -\frac{1}{1+z}\frac{1}{1+w}C_{a,b}\left(c^{\inv}_a(z), c^{\inv}_b(w)\right)\right]}\right) \\
&=\frac{w\left[(z+1) - (w+1) - (z+1) + C_{a,b}\left(c^{\inv}_a(z), c^{\inv}_b(w)\right)\right]}{z\left[(w+1) - (z+1) - (w+1) + C_{a,b}\left(c^{\inv}_a(z), c^{\inv}_b(w)\right)\right]}   \\
&=\frac{w\left[(z+1) - (w+1) - (z+1) + \left(1 + z + w + K_{a,b}\left( c^{\inv}_a(z),  c^{\inv}_b(w)\right)\right)\right]}{z\left[(w+1) - (z+1) - (w+1) + \left(1 + z + w + K_{a,b}\left( c^{\inv}_a(z),  c^{\inv}_b(w)\right)\right)\right]}\\
&=\frac{w\left[z + K_{a,b}\left( c^{\inv}_a(z),  c^{\inv}_b(w)\right)\right]}{z\left[ w + K_{a,b}\left( c^{\inv}_a(z),  c^{\inv}_b(w)\right)\right]}\\ 
&= \frac{1 + \frac{1}{z} K_{a,b}\left( c^{\inv}_a(z),  c^{\inv}_b(w)\right)}{1 + \frac{1}{w} K_{a,b}\left( c^{\inv}_a(z),  c^{\inv}_b(w)\right)}. \qedhere
\end{align*}
\end{proof}

\begin{rem}
\label{rem:assume-unital}
One might be concerned that we have restricted to the case $\varphi(a) = \varphi(b) = 1$.  However, if we use equation (\ref{eq:my-S-op}) as the definition of the opposite bi-free partial $S$-transform and if $\lambda, \mu \in \bC \setminus \{0\}$, then $S^{\op}_{\lambda a,\mu b}(z,w) = S^{\op}_{a, b}(z,w)$.  Indeed $c_{\lambda a}(z) = c_a(\lambda z)$ so $c_{\lambda a}^{\inv}(z) = \frac{1}{\lambda} c^{\inv}_a(z)$ and similarly $c_{\mu b}^{\inv}(w) = \frac{1}{\mu} c^{\inv}_b(w)$.  Therefore, since $\kappa_{n,m}(\lambda a, \mu b) = \lambda^n \mu^m \kappa_{n,m}(a,b)$, we see that 
\[
K_{\lambda a,\mu b}\left(c^{\inv}_{\lambda a}(z),  c^{\inv}_{\mu b}(w) \right) = K_{a,b}\left(c^{\inv}_{a}(z),  c^{\inv}_b(w) \right).
\]
Thus there is no loss in assuming $\varphi(a) = \varphi(b) = 1$.
\end{rem}

\begin{rem}
Note Proposition \ref{prop:S-op-new-defn} easily describes the bi-free partial $S^\op$-transform when $a$ and $b$ are independent; that is, when $\varphi(a^nb^m) = \varphi(a^n)\varphi(b^m)$ for all $n,m \geq 0$.  Indeed, in this setting $\kappa_{n,m}(a,b) = 0$ for all $n,m\geq 1$ (see \cite{S2016-2}*{Section 3.2}).  Hence $K_{a,b}(z,w) = 0$ so $S^\op_{a,b}(z,w) = 1$.  This agrees with the value of $S_{a,b}(z,w)$ when $a$ and $b$ are independent (see  \cite{V2016-2}*{Proposition 4.2} or \cite{S2016-1}*{Remark 4.4}).
\end{rem}

\begin{rem}
It is also curious to note that $\frac{1}{S^\op_{(a,b)}(z,w)}$ is well defined and would be an equivalent definition of the opposite bi-free partial $S$-transform.  Moreover, $S_{a,b}(z,w)$ and $S^\op_{a,b}(z,w)$ behave very differently with respect to interchanging $a$ and $b$.  Indeed it is elementary to see that
\[
S_{b,a}(z,w) = S_{a,b}(w,z) \qquad \text{whereas} \qquad S^\op_{b,a}(z,w) =  \frac{1}{S^\op_{a,b}(w,z)}.
\]
\end{rem}

Of course our goal is to prove the following desired property for $S^\op_{a,b}(z,w)$ using techniques similar to \cite{S2016-1}*{Theorem 4.5}.

\begin{thm}[\cite{HW2017}*{Proposition 2.8}]
\label{thm:S-op-property}
Let $(a_1, b_1)$ and $(a_2, b_2)$ be bi-free two-faced pairs in a Banach non-commutative probability space $(\A, \varphi)$ with $\varphi(a_k) \neq 0$ and $\varphi(b_k) \neq 0$.  Then
\[
S^{\op}_{a_1a_2, b_2b_1}(z,w) = S^\op_{a_1, b_1}(z,w) S^\op_{a_2, b_2}(z,w)
\]
on $(\bC \setminus \{0\})^2$ near $(0,0)$.
\end{thm}

\section{The Tools}
\label{sec:Tools}

Before we can prove Theorem \ref{thm:S-op-property}, we require some technical tools that are used in the proof.

\subsection{Multiplicative Functions and Convolutions}

Recall $NC(n)$ denotes the lattice of non-crossing partitions on $\{1,\ldots, n\}$ with its usual reverse refinement order.  Let $0_n$ to denote the minimal element of $NC(n)$ and let $1_n$ to denote the maximal element of $NC(n)$ as before.  For $\pi, \sigma \in NC(n)$ with $\pi \leq \sigma$, the interval between $\pi$ and $\sigma$, denoted $[\pi, \sigma]$, is the set
\[
[\pi, \sigma] = \{ \rho \in NC(n) \, \mid \, \pi \leq \rho \leq \sigma\}.
\]
The \emph{incidence algebra of non-crossing partitions}, denoted $\I(NC)$, is the algebra of all functions
\[
f : \bigcup_{n\geq 1} NC(n) \times NC(n) \to \bC
\]
such that $f(\pi, \sigma) = 0$ unless $\pi \leq \sigma$, equipped with pointwise addition and a convolution product defined by
\[
(f \ast g)(\pi, \sigma) = \sum_{\rho \in [\pi, \sigma]} f(\pi, \rho) g(\rho, \sigma).
\]

A procedure described in \cite{S1994} demonstrates a method for decomposing each interval $[\pi, \sigma]$ into a product of full partitions of the form
\[
[0_1, 1_1]^{k_1} \times [0_2, 1_2]^{k_2} \times [0_3, 1_3]^{k_3} \times \cdots
\]
where $k_j \geq 0$.  Consequently, there are special elements of the incidence algebra.  Recall $f \in \I(NC)$ is called \emph{multiplicative} if whenever $[\pi, \sigma]$ has a canonical decomposition $[0_1, 1_1]^{k_1} \times [0_2, 1_2]^{k_2} \times [0_3, 1_3]^{k_3} \times \cdots$, then
\[
f(\pi, \sigma) = f(0_1, 1_1)^{k_1} f(0_2, 1_2)^{k_2} f(0_3, 1_3)^{k_3} \cdots.
\]
Thus the value of a multiplicative function $f$ on any pair of non-crossing partitions is completely determined by the values of $f$ on full non-crossing partition lattices.  We will denote the set of all multiplicative functions by $\M$ and the set all multiplicative functions $f$ with $f(0_1, 1_1) = 1$ by $\M_1$.  If $f, g \in \M$, one can verify that $f \ast g = g \ast f$.  Furthermore, there is a nicer expression for convolution of multiplicative functions.  Given a non-crossing partition $\pi \in NC(n)$, the \emph{Kreweras complement of $\pi$}, denoted $K(\pi)$, is the non-crossing partition on $\{1, \ldots, n\}$ with non-crossing diagram obtained by drawing $\pi$ via the standard non-crossing diagram on $\{1,\ldots, n\}$, placing nodes $1', 2', \ldots, n'$ with $k'$ directly to the right of $k$, and drawing the largest non-crossing partition on $1', 2', \ldots, n'$ that does not intersect $\pi$, which is then $K(\pi)$.  The following diagram exhibits that if $\pi = \{\{1,4\}, \{2,3\}, \{5\}, \{6, 7\}\}$, then $K(\pi) = \{\{1,3\}, \{2\}, \{4,5,7\}, \{6\}\}$.
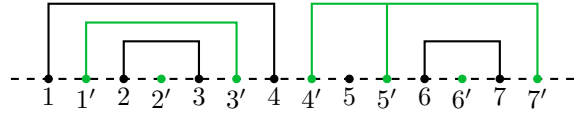
\begin{align*}
	\begin{tikzpicture}[baseline]
	\draw[thick] (1,0) -- (1,1) -- (4,1) -- (4,0);
	\draw[thick] (2,0) -- (2,.5) -- (3,.5) -- (3,0);
	\draw[thick] (6,0) -- (6,.5) -- (7,.5) -- (7,0);
	\draw[thick,ggreen] (4.5,0) -- (4.5,1) -- (7.5,1)--(7.5, 0);
	\draw[thick,ggreen] (5.5,0) -- (5.5,1);
	\draw[thick,ggreen] (1.5,0) -- (1.5,.75) -- (3.5,.75)--(3.5, 0);
	\draw[thick, dashed] (0.5,0) -- (8,0);
	\node[below] at (1, 0) {1};
	\draw[fill=black] (1,0) circle (0.05);
	\node[below] at (2, 0) {2};
	\draw[fill=black] (2,0) circle (0.05);
	\node[below] at (3, 0) {3};
	\draw[fill=black] (3,0) circle (0.05);
	\node[below] at (4, 0) {4};
	\draw[fill=black] (4,0) circle (0.05);
	\node[below] at (5, 0) {5};
	\draw[fill=black] (5,0) circle (0.05);
	\node[below] at (6, 0) {6};
	\draw[fill=black] (6,0) circle (0.05);
	\node[below] at (7, 0) {7};
	\draw[fill=black] (7,0) circle (0.05);
	\node[below] at (1.5, 0) {$1'$};
	\draw[ggreen, fill=ggreen] (1.5,0) circle (0.05);
	\node[below] at (2.5, 0) {$2'$};
	\draw[ggreen, fill=ggreen] (2.5,0) circle (0.05);
	\node[below] at (3.5, 0) {$3'$};
	\draw[ggreen, fill=ggreen] (3.5,0) circle (0.05);
	\node[below] at (4.5, 0) {$4'$};
	\draw[ggreen, fill=ggreen] (4.5,0) circle (0.05);
	\node[below] at (5.5, 0) {$5'$};
	\draw[ggreen, fill=ggreen] (5.5,0) circle (0.05);
	\node[below] at (6.5, 0) {$6'$};
	\draw[ggreen, fill=ggreen] (6.5,0) circle (0.05);
	\node[below] at (7.5, 0) {$7'$};
	\draw[ggreen, fill=ggreen] (7.5,0) circle (0.05);
	\end{tikzpicture}
\end{align*}
For $f, g \in \M$, the convolution $f \ast g$ may be written as
\[
(f \ast g)(0_n, 1_n) = \sum_{\pi \in NC(n)} f(0_n, \pi) g(0_n, K(\pi)).
\]
Note \cite{NS1997} demonstrated that if $a_1, a_2 \in \A$ are free and if $f_k$ is the multiplicative function associated to the cumulants of $a_k$ defined by $f_k(0_n, 1_n) = \kappa_n(a_k)$, then $\kappa_n(a_1a_2) = \kappa_n(a_2a_1) = (f_1 \ast f_2)(0_n, 1_n)$.  Furthermore, for $\pi \in NC(n)$ with blocks $\{V_k\}^m_{k=1}$, $f_k(0_n, \pi) = \kappa_\pi(a_k) = \prod^m_{k=1} \kappa_{|V_k|}(a_k)$.

We will need another convolution product on $\M_1$ from \cite{NS1997}.  Let $NC'(n)$ denote all non-crossing partitions $\pi$ on $\{1,\ldots, n\}$ such that $\{1\}$ is a block in $\pi$. It is not difficult to construct an natural isomorphism between $NC'(n)$ and $NC(n-1)$.  The following diagrams illustrate all elements $NC'(4)$, together with their Kreweras complements.
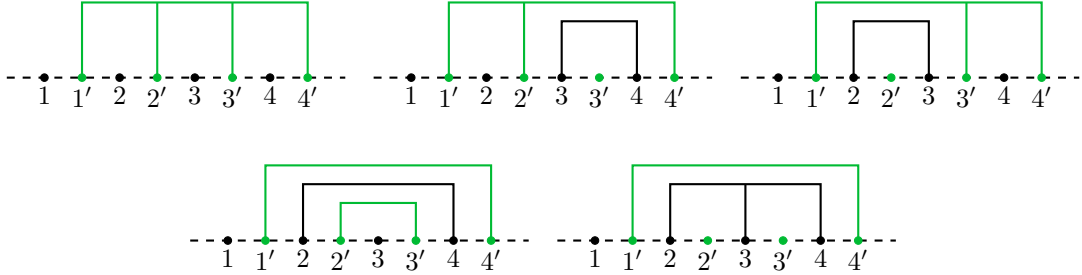
\begin{align*}
	\begin{tikzpicture}[baseline]
	\draw[thick,ggreen] (1.5,0) -- (1.5,1) -- (4.5,1)--(4.5, 0);
	\draw[thick,ggreen] (2.5,0) -- (2.5,1);
	\draw[thick,ggreen] (3.5,0) -- (3.5,1);
	\draw[thick, dashed] (0.5,0) -- (5,0);
	\node[below] at (1, 0) {1};
	\draw[fill=black] (1,0) circle (0.05);
	\node[below] at (2, 0) {2};
	\draw[fill=black] (2,0) circle (0.05);
	\node[below] at (3, 0) {3};
	\draw[fill=black] (3,0) circle (0.05);
	\node[below] at (4, 0) {4};
	\draw[fill=black] (4,0) circle (0.05);
	\node[below] at (1.5, 0) {$1'$};
	\draw[ggreen, fill=ggreen] (1.5,0) circle (0.05);
	\node[below] at (2.5, 0) {$2'$};
	\draw[ggreen, fill=ggreen] (2.5,0) circle (0.05);
	\node[below] at (3.5, 0) {$3'$};
	\draw[ggreen, fill=ggreen] (3.5,0) circle (0.05);
	\node[below] at (4.5, 0) {$4'$};
	\draw[ggreen, fill=ggreen] (4.5,0) circle (0.05);
	\end{tikzpicture} \quad
	\begin{tikzpicture}[baseline]
	\draw[thick,ggreen] (1.5,0) -- (1.5,1) -- (4.5,1)--(4.5, 0);
	\draw[thick,ggreen] (2.5,0) -- (2.5,1);
	\draw[thick,black] (3,0) -- (3, .75) -- (4, .75) -- (4,0);
	\draw[thick, dashed] (0.5,0) -- (5,0);
	\node[below] at (1, 0) {1};
	\draw[fill=black] (1,0) circle (0.05);
	\node[below] at (2, 0) {2};
	\draw[fill=black] (2,0) circle (0.05);
	\node[below] at (3, 0) {3};
	\draw[fill=black] (3,0) circle (0.05);
	\node[below] at (4, 0) {4};
	\draw[fill=black] (4,0) circle (0.05);
	\node[below] at (1.5, 0) {$1'$};
	\draw[ggreen, fill=ggreen] (1.5,0) circle (0.05);
	\node[below] at (2.5, 0) {$2'$};
	\draw[ggreen, fill=ggreen] (2.5,0) circle (0.05);
	\node[below] at (3.5, 0) {$3'$};
	\draw[ggreen, fill=ggreen] (3.5,0) circle (0.05);
	\node[below] at (4.5, 0) {$4'$};
	\draw[ggreen, fill=ggreen] (4.5,0) circle (0.05);
	\end{tikzpicture}
	\quad
	\begin{tikzpicture}[baseline]
	\draw[thick,ggreen] (1.5,0) -- (1.5,1) -- (4.5,1)--(4.5, 0);
	\draw[thick,black] (2,0) -- (2,.75) -- (3,.75) -- (3,0);
	\draw[thick,ggreen] (3.5,0) -- (3.5,1);
	\draw[thick, dashed] (0.5,0) -- (5,0);
	\node[below] at (1, 0) {1};
	\draw[fill=black] (1,0) circle (0.05);
	\node[below] at (2, 0) {2};
	\draw[fill=black] (2,0) circle (0.05);
	\node[below] at (3, 0) {3};
	\draw[fill=black] (3,0) circle (0.05);
	\node[below] at (4, 0) {4};
	\draw[fill=black] (4,0) circle (0.05);
	\node[below] at (1.5, 0) {$1'$};
	\draw[ggreen, fill=ggreen] (1.5,0) circle (0.05);
	\node[below] at (2.5, 0) {$2'$};
	\draw[ggreen, fill=ggreen] (2.5,0) circle (0.05);
	\node[below] at (3.5, 0) {$3'$};
	\draw[ggreen, fill=ggreen] (3.5,0) circle (0.05);
	\node[below] at (4.5, 0) {$4'$};
	\draw[ggreen, fill=ggreen] (4.5,0) circle (0.05);
	\end{tikzpicture}
\end{align*}
\begin{align*}
	\begin{tikzpicture}[baseline]
	\draw[thick,ggreen] (1.5,0) -- (1.5,1) -- (4.5,1)--(4.5, 0);
	\draw[thick,ggreen] (2.5,0) -- (2.5,.5) -- (3.5, .5) -- (3.5,0);
	\draw[thick,black] (2,0) -- (2,.75) -- (4, .75) -- (4,0);
	\draw[thick, dashed] (0.5,0) -- (5,0);
	\node[below] at (1, 0) {1};
	\draw[fill=black] (1,0) circle (0.05);
	\node[below] at (2, 0) {2};
	\draw[fill=black] (2,0) circle (0.05);
	\node[below] at (3, 0) {3};
	\draw[fill=black] (3,0) circle (0.05);
	\node[below] at (4, 0) {4};
	\draw[fill=black] (4,0) circle (0.05);
	\node[below] at (1.5, 0) {$1'$};
	\draw[ggreen, fill=ggreen] (1.5,0) circle (0.05);
	\node[below] at (2.5, 0) {$2'$};
	\draw[ggreen, fill=ggreen] (2.5,0) circle (0.05);
	\node[below] at (3.5, 0) {$3'$};
	\draw[ggreen, fill=ggreen] (3.5,0) circle (0.05);
	\node[below] at (4.5, 0) {$4'$};
	\draw[ggreen, fill=ggreen] (4.5,0) circle (0.05);
	\end{tikzpicture}
	\quad
	\begin{tikzpicture}[baseline]
	\draw[thick,ggreen] (1.5,0) -- (1.5,1) -- (4.5,1)--(4.5, 0);
	\draw[thick,black] (2,0) -- (2,.75) -- (4, .75) -- (4,0);
	\draw[thick,black] (3,0) -- (3,.75);
	\draw[thick, dashed] (0.5,0) -- (5,0);
	\node[below] at (1, 0) {1};
	\draw[fill=black] (1,0) circle (0.05);
	\node[below] at (2, 0) {2};
	\draw[fill=black] (2,0) circle (0.05);
	\node[below] at (3, 0) {3};
	\draw[fill=black] (3,0) circle (0.05);
	\node[below] at (4, 0) {4};
	\draw[fill=black] (4,0) circle (0.05);
	\node[below] at (1.5, 0) {$1'$};
	\draw[ggreen, fill=ggreen] (1.5,0) circle (0.05);
	\node[below] at (2.5, 0) {$2'$};
	\draw[ggreen, fill=ggreen] (2.5,0) circle (0.05);
	\node[below] at (3.5, 0) {$3'$};
	\draw[ggreen, fill=ggreen] (3.5,0) circle (0.05);
	\node[below] at (4.5, 0) {$4'$};
	\draw[ggreen, fill=ggreen] (4.5,0) circle (0.05);
	\end{tikzpicture}
\end{align*}
Notice if $\pi \in NC'(n)$ and $\sigma$ is the non-crossing partition on $\{1, 1', 2, 2', \ldots, n, n'\}$ (with the ordering being the order of listing) with blocks $\{k, k'\}$ for all $k$, then the only non-crossing partition $\tau$ on $\{1', \ldots, n'\}$ such that $\pi \cup \tau$ is non-crossing (under the ordering $1, 1', 2, 2', \ldots, n, n'$) and $(\pi \cup \tau) \vee \sigma = 1_{2n}$ is $\tau = K(\pi)$.  Indeed this result can be shown by induction.

For $f, g \in \M_1$ the \emph{pinched-convolution} of $f$ and $g$, denoted $f \check{\ast} g$, is the unique element of $\M_1$ such that
\[
(f \check{\ast} g)[0_n, 1_n] = \sum_{\pi \in NC'(n)} f(0_n, \pi) g(0_n, K(\pi)).
\]
The pinched-convolution product is not commutative on $\M_1$.

Given an element $f \in \M$, we define the formal power series
\[
\phi_f(z) := \sum_{n\geq 1} f(0_n, 1_n) z^n.
\]
Note if $f$ is the multiplicative function associated to the cumulants of $a$ defined by $f(0_n, 1_n) = \kappa_n(a)$, then $\phi_f(z) = c_a(z)$.  Several formulae involving $\phi_f(z)$ are developed in \cite{NS1997}.  In particular, \cite{NS1997}*{Proposition 2.3} demonstrates that if $f_1, f_2 \in \M_1$ then $\phi_{f_1}(\phi_{f_1 \check{\ast} f_2}(z)) = \phi_{f_1 \ast f_2}(z)$ and thus
\begin{align}
\phi_{f_1 \check{\ast} f_2}\left(\phi^{\inv}_{f_1 \ast f_2}(z)\right) = \phi^{\inv}_{f_1}(z).  \label{eq:pinch-convolution-technicality}
\end{align}
Furthermore, \cite{NS1997}*{Theorem 1.6} demonstrates that
\begin{align}
z \cdot \phi^{\inv}_{f_1 \ast f_2}(z) = \phi^{\inv}_{f_1}(z)\phi^{\inv}_{f_2}(z).  \label{eq:convolution-with-inverse-series}
\end{align}
Note that equation \eqref{eq:S-cumulants-form} is an immediate consequence of equation (\ref{eq:convolution-with-inverse-series}) when  $\varphi(a) = 1$.

\subsection{Bi-Free Kreweras Complements}
\label{subsec:Bi-Krew}

The one additional technicality required revolves around the Kreweras complement for bi-non-crossing partitions described in \cite{CNS2015-2}.  To begin, recall if $\chi : \{1,\ldots, n\} \to \{\ell, r\}$ is such that
\[
\chi^{-1}(\{\ell\}) = \{k_1 < k_2 < \cdots < k_j\} \qqand \chi^{-1}(\{r\}) = \{k_{j+1} > k_{j+2} > \cdots > k_n\},
\]
there is a corresponding permutation $s_\chi$ on $\{1,\ldots, n\}$ defined by $s_\chi(j) = k_j$.  Consequently, $\pi \in BNC(\chi)$ if and only if $s^{-1}_\chi \cdot \pi$ (the partition obtained by applying $s^{-1}_\chi$ to each element of each block) is non-crossing on $\{1,\ldots, n\}$.

\begin{defn}
\label{defn:Kreweras}
For all $n \in \bN$, $\chi : \{1,\ldots, n\} \to \{\ell, r\}$, and $\pi \in BNC(\chi)$, the \emph{Kreweras complement} of $\pi$, denoted $K(\pi)$, is the element of $BNC(\chi)$ obtained by applying $s_\chi$ to the Kreweras complement in $NC(n)$ of $s^{-1}_\chi \cdot \pi$; that is, $K(\pi) = s_\chi \cdot K(s^{-1}_\chi \cdot \pi)$.
\end{defn}

\begin{rem}
\label{rem:Kreweras-BNC}
Note that $K(\pi)$ may be obtained by taking the diagram corresponding to $\pi$, placing a node beneath each left node and above each right node of $\pi$, and drawing the largest bi-non-crossing diagram on the new nodes.  Indeed the following diagrams illustrate this procedure.
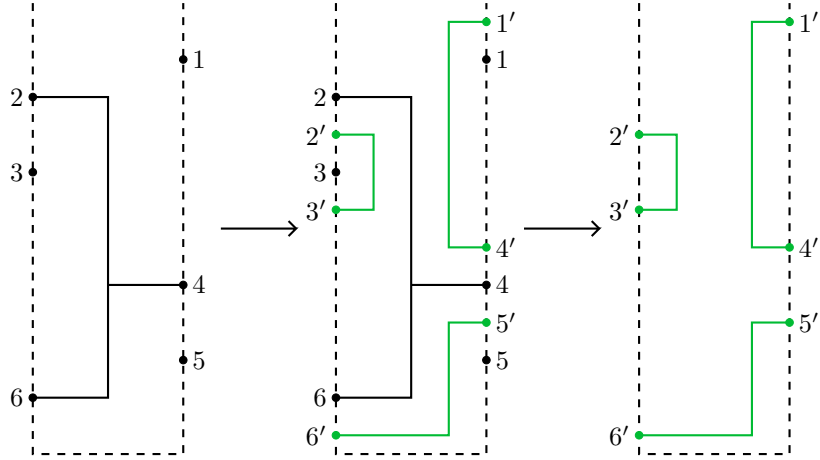
\begin{align*}
\begin{tikzpicture}[baseline]
	\draw[thick] (1.5, 2.75) -- (2.5, 2.75) -- (2.4, 2.65);
	\draw[thick] (2.5,2.75) -- (2.4, 2.85);
	\draw[thick, dashed] (-1,5.75) -- (-1,-.25) -- (1,-.25) -- (1,5.75);
	\draw[thick, black] (-1,4.5) -- (0, 4.5) -- (0, .5) -- (-1,.5);
	\draw[thick, black] (1,2) -- (0, 2);
	\node[right] at (1,5) {$1$};
	\draw[fill=black] (1,5) circle (0.05);
	\node[left] at (-1, 4.5) {$2$};
	\draw[fill=black] (-1,4.5) circle (0.05);
	\node[left] at (-1, 3.5) {$3$};
	\draw[fill=black] (-1,3.5) circle (0.05);;
	\node[right] at (1,2) {$4$};
	\draw[fill=black] (1,2) circle (0.05);
	\node[right] at (1,1) {$5$};
	\draw[fill=black] (1,1) circle (0.05);
	\node[left] at (-1,0.5) {$6$};
	\draw[fill=black] (-1,0.5) circle (0.05);
	\end{tikzpicture}
	\begin{tikzpicture}[baseline]
	\draw[thick] (1.5, 2.75) -- (2.5, 2.75) -- (2.4, 2.65);
	\draw[thick] (2.5,2.75) -- (2.4, 2.85);
	\draw[thick, dashed] (-1,5.75) -- (-1,-.25) -- (1,-.25) -- (1,5.75);
	\draw[thick, black] (-1,4.5) -- (0, 4.5) -- (0, .5) -- (-1,.5);
	\draw[thick, black] (1,2) -- (0, 2);
	\draw[thick, ggreen] (-1,4) -- (-.5, 4) -- (-.5, 3) -- (-1,3);
	\draw[thick, ggreen] (1,5.5) -- (.5, 5.5) -- (.5, 2.5) -- (1,2.5);
	\draw[thick, ggreen] (-1,0) -- (.5, 0) -- (.5, 1.5) -- (1,1.5);
	\node[right] at (1,5.5) {$1'$};
	\draw[ggreen, fill=ggreen] (1,5.5) circle (0.05);
	\node[right] at (1,5) {$1$};
	\draw[fill=black] (1,5) circle (0.05);
	\node[left] at (-1, 4.5) {$2$};
	\draw[fill=black] (-1,4.5) circle (0.05);
	\node[left] at (-1, 4) {$2'$};
	\draw[ggreen, fill=ggreen] (-1,4) circle (0.05);
	\node[left] at (-1, 3.5) {$3$};
	\draw[fill=black] (-1,3.5) circle (0.05);
	\node[left] at (-1, 3) {$3'$};
	\draw[ggreen, fill=ggreen] (-1,3) circle (0.05);
	\node[right] at (1,2.5) {$4'$};
	\draw[ggreen, fill=ggreen] (1,2.5) circle (0.05);
	\node[right] at (1,2) {$4$};
	\draw[fill=black] (1,2) circle (0.05);
	\node[right] at (1,1.5) {$5'$};
	\draw[ggreen, fill=ggreen] (1,1.5) circle (0.05);
	\node[right] at (1,1) {$5$};
	\draw[fill=black] (1,1) circle (0.05);
	\node[left] at (-1,0.5) {$6$};
	\draw[fill=black] (-1,0.5) circle (0.05);
	\node[left] at (-1,0) {$6'$};
	\draw[ggreen, fill=ggreen] (-1,0) circle (0.05);
	\end{tikzpicture}
	\begin{tikzpicture}[baseline]
	\draw[thick, dashed] (-1,5.75) -- (-1,-.25) -- (1,-.25) -- (1,5.75);
	\draw[thick, ggreen] (-1,4) -- (-.5, 4) -- (-.5, 3) -- (-1,3);
	\draw[thick, ggreen] (1,5.5) -- (.5, 5.5) -- (.5, 2.5) -- (1,2.5);
	\draw[thick, ggreen] (-1,0) -- (.5, 0) -- (.5, 1.5) -- (1,1.5);
	\node[right] at (1,5.5) {$1'$};
	\draw[ggreen, fill=ggreen] (1,5.5) circle (0.05);
	\node[left] at (-1, 4) {$2'$};
	\draw[ggreen, fill=ggreen] (-1,4) circle (0.05);
	\node[left] at (-1, 3) {$3'$};
	\draw[ggreen, fill=ggreen] (-1,3) circle (0.05);
	\node[right] at (1,2.5) {$4'$};
	\draw[ggreen, fill=ggreen] (1,2.5) circle (0.05);
	\node[right] at (1,1.5) {$5'$};
	\draw[ggreen, fill=ggreen] (1,1.5) circle (0.05);
	\node[left] at (-1,0) {$6'$};
	\draw[ggreen, fill=ggreen] (-1,0) circle (0.05);
	\end{tikzpicture}
\end{align*}
Furthermore, given a $\pi \in BNC(n,m)$, the above demonstrates a natural way to view the pair $(\pi, K(\pi))$ as an element of $BNC(2n, 2m)$.  Indeed if one applies the map 
\[
(1_\ell, \ldots, n_\ell, 1_r, \ldots, m_r) \mapsto (1_\ell, 3_\ell, \ldots, (2n-1)_\ell, 2_r, 4_r, \ldots, (2m)_r)
\]
to the blocks in $\pi$ and the map
\[
(1_\ell, \ldots, n_\ell, 1_r, \ldots, m_r) \mapsto (2_\ell, 4_\ell, \ldots, (2n)_\ell, 1_r, 3_r, \ldots, (2m-1)_r)
\]
to the blocks in $K(\pi)$, then $\pi \cup K(\pi)$ is an element of $BNC(2n, 2m)$.  Note an element $\sigma \in BNC(2n, 2m)$ can be obtain in this fashion if and only if $\sigma \vee \sigma_{n,m} = 1_{2n,2m}$ where 
\[
\sigma_{n,m} = \{\{(2i-1)_\ell, (2i)_\ell\}\}^n_{i=1} \cup \{\{(2j-1)_r, (2j)_r\}\}^m_{j=1}
\]
and $\sigma$ contains no blocks with a $(2p)_\theta$ and a $(2q-1)_\theta$ for any $p$, $q$, and $\theta \in \{\ell,r\}$, no blocks with a $(2p)_\ell$ and a $(2q)_r$ for any $p$ and $q$, and no blocks with a $(2p-1)_\ell$ and a $(2q-1)_r$ for any $p$ and $q$.
\end{rem}

The reason Kreweras complements of bi-non-crossing partitions are essential in this paper is the following.

\begin{thm}[\cite{CNS2015-2}*{Theorem 5.2.1}]
\label{thm:cumulants-and-Kreweras}
Let $(a_1, b_1)$ and $(a_2, b_2)$ be bi-free two-faced pairs in a non-commutative probability space $(\A, \varphi)$.  Then for all $n,m\geq 0$
\[
\kappa_{n,m}(a_1a_2, b_2b_1) = \sum_{\pi \in BNC(n,m)} \kappa_{\pi}(a_1, b_1) \kappa_{K(\pi)}(a_2, b_2).
\]
\end{thm}

\section{The Proof}
\label{sec:Proof}

To simplify the proof of Theorem \ref{thm:S-op-property}, we will assume via Remark \ref{rem:assume-unital} that $\varphi(a_k) = \varphi(b_k) = 1$.  Note $\varphi(a_1a_2) = \varphi(b_2b_1) = 1$ by freeness of the left algebras and of the right algebras in bi-free pairs.  Furthermore, we will let $f_k$ (respectively $g_k$) denote the multiplicative function associated to the cumulants of $a_k$ (respectively $b_k$) defined by $f_k(0_n, 1_n) = \kappa_n(a_k)$ (respectively $g_k(0_n, 1_n) = \kappa_n(b_k)$).  Recall if $f$ (respectively $g$) is the multiplicative function associated to the cumulants of $a_1a_2$ (respectively $b_2b_1$), then $f = f_1 \ast f_2$ (respectively $g = g_2 \ast g_1 = g_1 \ast g_2$).  Thus $\phi^{\inv}_{f}(z) = c^{\inv}_{a_1a_2}(z)$, $\phi^{\inv}_g(w) = c^{\inv}_{b_2b_1}(w)$, $\phi^{\inv}_{f_k}(z) = c^{\inv}_{a_k}(z)$, and $\phi^{\inv}_{g_k}(w) = c^{\inv}_{b_k}(w)$.   Note that $f, g, f_k, g_k \in \M_1$ by assumptions.

The majority of the proof of Theorem \ref{thm:S-op-property} is to show the following.
\begin{lem}
\label{lem:K-expression}
Let $(a_1, b_1)$ and $(a_2, b_2)$ be bi-free two-faced pairs in a non-commutative probability space $(\A, \varphi)$ with $\varphi(a_k) =1$ and $\varphi(b_k) =1$.  Then, on $(\bC \setminus \{0\})^2$ near $(0,0)$,
\begin{align}
K_{a_1a_2,b_2b_1}\left( c^{\inv}_{a_1a_2}(z),  c^{\inv}_{b_2b_1}(w)\right) = \frac{Q_1(z,w) + Q_2(z,w) + \left(\frac{1}{z} + \frac{1}{w}\right) Q_1(z,w)Q_2(z,w)}{1 - \frac{1}{zw}Q_1(z,w)Q_2(z,w)} \label{eq:main-technical-formula}
\end{align}
where
\begin{align}
Q_k(z,w) = K_{a_k,b_k}\left( c^{\inv}_{a_k}(z),  c^{\inv}_{b_k}(w)\right). \label{eq:Q-formula}
\end{align}
\end{lem}

Again note the expression in equation \eqref{eq:main-technical-formula} is well defined as if $(z,w)$ are near $(0,0)$ then $\frac{1}{zw}Q_k(z,w)$ is well defined and near 1.

To begin the proof of Lemma \ref{lem:K-expression}, recall
\[
K_{a_1a_2, b_2b_1}(z,w) = \sum_{n,m\geq 1} \kappa_{n,m}(a_1a_2, b_2b_1) z^n w^m.
\]
Therefore, since $(a_1, b_1)$ and $(a_2, b_2)$ are bi-free, Theorem \ref{thm:cumulants-and-Kreweras} implies that
\[
\kappa_{n,m}(a_1a_2, b_2b_1) = \sum_{ \pi \in BNC(n,m)} \kappa_\pi(a_1, b_1)\kappa_{K(\pi)}(a_2, b_2).
\]
It is necessary to analyze and decompose the pair $(\pi, K(\pi))$.

Given a $\pi \in BNC(n,m)$, we will view the pair $(\pi, K(\pi))$ as an element of $\sigma \in BNC(2n, 2m)$ as in Remark \ref{rem:Kreweras-BNC}.  Under this view, there are two possible cases: the block of $\pi$ containing $1_\ell$ contains a $(2j)_r$ for some $j$, or the block of $K(\pi)$ containing $1_r$ contains a $(2i)_\ell$ for some $i$.  Indeed if the block of $\pi$ containing $1_\ell$ does not contain a $(2j)_r$ for some $j$ and $i$ is the largest integer such that the block of $\pi$ containing $1_\ell$ also contains $(2i-1)_\ell$, then by the properties of the Kreweras complement the block of $K(\pi)$ containing $1_r$ contains $(2i)_\ell$.  Consequently, if
\begin{align*}
BNC_L(n,m) &= \{ \pi \in BNC(n,m) \, \mid \, \text{the block of $\pi$ containing $1_\ell$ contains a $(2j)_r$ for some $j$}\} \\
BNC_R(n,m) &= \{ \pi \in BNC(n,m) \, \mid \, \text{the block of $K(\pi)$ containing $1_r$ contains a $(2i)_\ell$ for some $k$}\}
\end{align*}
then $BNC(n,m)$ is the disjoint union of $BNC_L(n,m)$ and $BNC_R(n,m)$.  

In order to discuss and decompose elements of $BNC_L(n,m)$ and $BNC_R(n,m)$, it is necessary to discuss other bi-non-crossing partitions.  Given $n,m \geq 0$, let $\sigma_L$ denote the element of $BNC(2n+1, 2m)$ with blocks $\{\{(2i)_\ell, (2i+1)_\ell\}\}_{i=1}^{n} \cup \{\{(2j-1)_r, (2j)_r\}\}_{j=1}^{m}$ and let $\sigma_R$ denote the element of $BNC(2n, 2m+1)$ with blocks $\{\{(2i-1)_\ell, (2i)_\ell\}\}_{i=1}^{n} \cup \{\{(2j)_r, (2j+1)_r\}\}_{j=1}^{m}$.  Then, for all $n,m \geq 0$ let
\begin{align*}
BNC_{Lb}(2n+1, 2m) &= \left\{\pi \in BNC(2n+1, 2m) \, \left| \, \substack{ \pi \vee \sigma_L = 1_{2n+1, 2m} \text{ and no block of $\pi$ contains a} \\ (2k)_{\theta_1} \text{ and a } (2j-1)_{\theta_2} \text{ for any } \theta_1, \theta_2 \in \{l, r\}} \right. \right\} \\
BNC_{Rb}(2n, 2m+1) &= \left\{\pi \in BNC(2n, 2m+1) \, \left| \, \substack{ \pi \vee \sigma_R = 1_{2n, 2m+1} \text{ and no block of $\pi$ contains a} \\ (2k)_{\theta_1} \text{ and a } (2j-1)_{\theta_2} \text{ for any } \theta_1, \theta_2 \in \{l, r\}} \right. \right\}.
\end{align*}
As will be seen in the proofs which follow, elements of $BNC_{Lb}(2n'+1, 2m')$ are the `bottoms' of elements of $BNC_L(n,m)$ and elements of $BNC_{Rb}(2n', 2m'+1)$ are the `bottoms' of elements of $BNC_R(n,m)$.

Using these bi-non-crossing partitions, we define the following transformations:
\begin{align*}
\phi_L(z,w) &:= \sum_{n,m\geq 1} \sum_{\pi \in BNC_L(n,m)} \kappa_\pi(a_1, b_1)\kappa_{K(\pi)}(a_2, b_2) z^n w^m \\
\phi_R(z,w) &:= \sum_{n,m\geq 1} \sum_{\pi \in BNC_R(n,m)} \kappa_\pi(a_1, b_1)\kappa_{K(\pi)}(a_2, b_2) z^n w^m \\ 
\psi_L(z,w) &:= \sum_{n,m\geq 0}\sum_{\pi \in BNC_{Lb}(2n+1,2m)} \kappa_\pi(\underbrace{a_2, a_1, a_2, a_1 \ldots, a_2, a_1, a_2}_{a_1 \text{ occurs }n \text{ times}}, \underbrace{b_2, b_1, b_2, b_1, \ldots, b_2, b_1}_{b_1, b_2 \text{ occur }m \text{ times}} ) z^{n+1} w^m \\
\psi_R(z,w) &:= \sum_{n,m\geq 0}\sum_{\pi \in BNC_{Rb}(2n,2m+1)} \kappa_\pi(\underbrace{a_1, a_2, a_1, a_2 \ldots, a_1, a_2}_{a_1, a_2 \text{ occur }n \text{ times}}, \underbrace{b_1, b_2, b_1, b_2, \ldots, b_1, b_2, b_1}_{b_2 \text{ occurs }m \text{ times}} ) z^{n} w^{m+1}.
\end{align*}
Notice that each of the above is a well-defined holomorphic function near $(0,0)$ and that
\begin{align}
K_{a_1a_2, b_2b_1}(z,w) = \phi_L(z,w) + \phi_R(z,w) \label{eq:decompose-K}
\end{align}
by construction.  Thus to understand $K_{a_1a_2, b_2b_1}(z,w)$ we must understand $\phi_L(z,w)$ and $\phi_R(z,w)$. The reason we must consider $\psi_L(z,w)$ and $\psi_R(z,w)$ are the following.

\begin{lem}
\label{lem:1}
Under the above notation and assumptions,
\[
\phi_L(z,w) = \frac{1}{\phi_{f_1  \check{\ast} f_2}(z)} K_{a_1, b_1}\left(\phi_{f_1  \check{\ast} f_2}(z), \phi_{g_1  \check{\ast} g_2}(w)\right) \psi_L(z,w)
\]
as holomorphic functions near $(0,0)$.
\end{lem}
\begin{proof}
First observe that dividing by $\phi_{f_1  \check{\ast} f_2}(z)$ is valid (e.g. see the end of the proof for how this division term occurs).  Furthermore all terms involved in proof involve absolutely summable series and thus can be rearranged in any order we desired.

For now fix $n,m \geq 1$ and a $\pi \in BNC_L(n,m)$.  View $(\pi, K(\pi))$ as an element of $\sigma \in BNC(2n, 2m)$ as in Remark \ref{rem:Kreweras-BNC}.  Below is two diagrams of such a $\pi$ with $n = 5$ and $m = 6$.  The left diagram illustrates how the bi-non-crossing partition is the union of $\pi$ and $K(\pi)$ via shading.  The diagram on the right highlights portions of the diagram used in the proof.  Note two nodes are connected to each other with a solid line if and only if lie in the same block of $\pi$.  Furthermore, due to properties of the Kreweras complement, one must be able travel from any one node to any another using a combination of solid  and dotted lines. Note we really should draw all of the left nodes above all of the right notes, but we do not do so in order to save space.

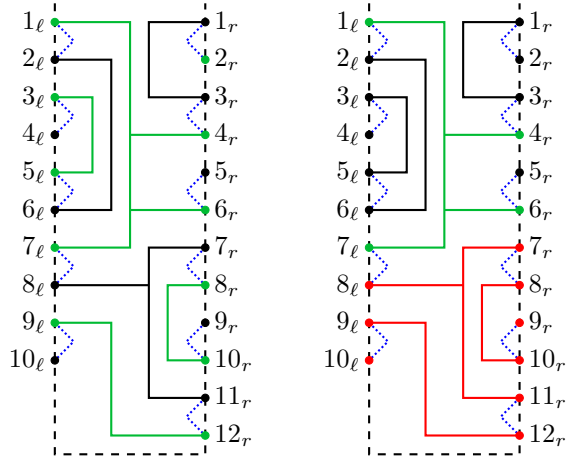
\begin{align*}
\begin{tikzpicture}[baseline]
	\draw[thick, dashed] (-1,5.75) -- (-1,-.25) -- (1,-.25) -- (1,5.75);
	\draw[thick, blue, densely dotted] (1, 5.5) -- (0.75, 5.25) -- (1, 5);
	\draw[thick, blue, densely dotted] (1, 4.5) -- (0.75, 4.25) -- (1, 4);
	\draw[thick, blue, densely dotted] (1, 3.5) -- (0.75, 3.25) -- (1, 3);
	\draw[thick, blue, densely dotted] (1, 2.5) -- (0.75, 2.25) -- (1, 2);
	\draw[thick, blue, densely dotted] (1, 1.5) -- (0.75, 1.25) -- (1, 1);
	\draw[thick, blue, densely dotted] (1, 0.5) -- (0.75, 0.25) -- (1, 0);
	\draw[thick, blue, densely dotted] (-1, 5.5) -- (-0.75, 5.25) -- (-1, 5);
	\draw[thick, blue, densely dotted] (-1, 4.5) -- (-0.75, 4.25) -- (-1, 4);
	\draw[thick, blue, densely dotted] (-1, 3.5) -- (-0.75, 3.25) -- (-1, 3);
	\draw[thick, blue, densely dotted] (-1, 2.5) -- (-0.75, 2.25) -- (-1, 2);
	\draw[thick, blue, densely dotted] (-1, 1.5) -- (-0.75, 1.25) -- (-1, 1);
	\draw[thick, ggreen] (-1,5.5) -- (0,5.5) -- (0,2.5) -- (-1,2.5);
	\draw[thick, ggreen] (1,4) -- (0,4);
	\draw[thick, ggreen] (1,3) -- (0,3);
	\draw[thick, black] (1,5.5) -- (0.25, 5.5) -- (0.25, 4.5) -- (1,4.5);
	\draw[thick, black] (-1,5) -- (-0.25, 5) -- (-0.25, 3) -- (-1,3);
	\draw[thick, ggreen] (-1,4.5) -- (-0.5, 4.5) -- (-0.5, 3.5) -- (-1,3.5);
	\draw[thick, black] (1,2.5) -- (0.25,2.5) -- (0.25,.5) -- (1,.5);
	\draw[thick, black] (-1,2) -- (0.25,2);
	\draw[thick, ggreen] (1,2) -- (0.5,2) -- (0.5,1) -- (1,1);
	\draw[thick, ggreen] (-1,1.5) -- (-.25,1.5) -- (-0.25,0) -- (1,0);
	\node[left] at (-1, 5.5) {$1_\ell$};
	\draw[ggreen, fill=ggreen] (-1,5.5) circle (0.05);
	\node[left] at (-1, 5) {$2_\ell$};
	\draw[fill=black] (-1,5) circle (0.05);
	\node[left] at (-1, 4.5) {$3_\ell$};
	\draw[ggreen, fill=ggreen] (-1,4.5) circle (0.05);
	\node[left] at (-1, 4) {$4_\ell$};
	\draw[fill=black] (-1,4) circle (0.05);
	\node[left] at (-1, 3.5) {$5_\ell$};
	\draw[ggreen, fill=ggreen] (-1,3.5) circle (0.05);
	\node[left] at (-1, 3) {$6_\ell$};
	\draw[fill=black] (-1,3) circle (0.05);
	\node[left] at (-1, 2.5) {$7_\ell$};
	\draw[ggreen, fill=ggreen] (-1,2.5) circle (0.05);
	\node[left] at (-1, 2) {$8_\ell$};
	\draw[fill=black] (-1,2) circle (0.05);
	\node[left] at (-1, 1.5) {$9_\ell$};
	\draw[ggreen, fill=ggreen] (-1,1.5) circle (0.05);
	\node[left] at (-1, 1) {$10_\ell$};
	\draw[fill=black] (-1,1) circle (0.05);
	\node[right] at (1,5.5) {$1_r$};
	\draw[fill=black] (1,5.5) circle (0.05);
	\node[right] at (1,5) {$2_r$};
	\draw[ggreen, fill=ggreen] (1,5) circle (0.05);
	\node[right] at (1,4.5) {$3_r$};
	\draw[fill=black] (1,4.5) circle (0.05);
	\node[right] at (1,4) {$4_r$};
	\draw[ggreen, fill=ggreen] (1,4) circle (0.05);
	\node[right] at (1,3.5) {$5_r$};
	\draw[fill=black] (1,3.5) circle (0.05);
	\node[right] at (1,3) {$6_r$};
	\draw[ggreen, fill=ggreen] (1,3) circle (0.05);
	\node[right] at (1,2.5) {$7_r$};
	\draw[fill=black] (1,2.5) circle (0.05);
	\node[right] at (1,2) {$8_r$};
	\draw[ggreen, fill=ggreen] (1,2) circle (0.05);
	\node[right] at (1,1.5) {$9_r$};
	\draw[fill=black] (1,1.5) circle (0.05);
	\node[right] at (1,1) {$10_r$};
	\draw[ggreen, fill=ggreen] (1,1) circle (0.05);
	\node[right] at (1,0.5) {$11_r$};
	\draw[fill=black] (1,0.5) circle (0.05);
	\node[right] at (1,0) {$12_r$};
	\draw[ggreen, fill=ggreen] (1,0) circle (0.05);
	\end{tikzpicture}
	\qquad
	\begin{tikzpicture}[baseline]
	\draw[thick, dashed] (-1,5.75) -- (-1,-.25) -- (1,-.25) -- (1,5.75);
	\draw[thick, blue, densely dotted] (1, 5.5) -- (0.75, 5.25) -- (1, 5);
	\draw[thick, blue, densely dotted] (1, 4.5) -- (0.75, 4.25) -- (1, 4);
	\draw[thick, blue, densely dotted] (1, 3.5) -- (0.75, 3.25) -- (1, 3);
	\draw[thick, blue, densely dotted] (1, 2.5) -- (0.75, 2.25) -- (1, 2);
	\draw[thick, blue, densely dotted] (1, 1.5) -- (0.75, 1.25) -- (1, 1);
	\draw[thick, blue, densely dotted] (1, 0.5) -- (0.75, 0.25) -- (1, 0);
	\draw[thick, blue, densely dotted] (-1, 5.5) -- (-0.75, 5.25) -- (-1, 5);
	\draw[thick, blue, densely dotted] (-1, 4.5) -- (-0.75, 4.25) -- (-1, 4);
	\draw[thick, blue, densely dotted] (-1, 3.5) -- (-0.75, 3.25) -- (-1, 3);
	\draw[thick, blue, densely dotted] (-1, 2.5) -- (-0.75, 2.25) -- (-1, 2);
	\draw[thick, blue, densely dotted] (-1, 1.5) -- (-0.75, 1.25) -- (-1, 1);
	\draw[thick, ggreen] (-1,5.5) -- (0,5.5) -- (0,2.5) -- (-1,2.5);
	\draw[thick, ggreen] (1,4) -- (0,4);
	\draw[thick, ggreen] (1,3) -- (0,3);
	\draw[thick, black] (1,5.5) -- (0.25, 5.5) -- (0.25, 4.5) -- (1,4.5);
	\draw[thick, black] (-1,5) -- (-0.25, 5) -- (-0.25, 3) -- (-1,3);
	\draw[thick, black] (-1,4.5) -- (-0.5, 4.5) -- (-0.5, 3.5) -- (-1,3.5);
	\draw[thick, red] (1,2.5) -- (0.25,2.5) -- (0.25,.5) -- (1,.5);
	\draw[thick, red] (-1,2) -- (0.25,2);
	\draw[thick, red] (1,2) -- (0.5,2) -- (0.5,1) -- (1,1);
	\draw[thick, red] (-1,1.5) -- (-.25,1.5) -- (-0.25,0) -- (1,0);
	\node[left] at (-1, 5.5) {$1_\ell$};
	\draw[ggreen, fill=ggreen] (-1,5.5) circle (0.05);
	\node[left] at (-1, 5) {$2_\ell$};
	\draw[fill=black] (-1,5) circle (0.05);
	\node[left] at (-1, 4.5) {$3_\ell$};
	\draw[fill=black] (-1,4.5) circle (0.05);
	\node[left] at (-1, 4) {$4_\ell$};
	\draw[fill=black] (-1,4) circle (0.05);
	\node[left] at (-1, 3.5) {$5_\ell$};
	\draw[fill=black] (-1,3.5) circle (0.05);
	\node[left] at (-1, 3) {$6_\ell$};
	\draw[fill=black] (-1,3) circle (0.05);
	\node[left] at (-1, 2.5) {$7_\ell$};
	\draw[ggreen, fill=ggreen] (-1,2.5) circle (0.05);
	\node[left] at (-1, 2) {$8_\ell$};
	\draw[red, fill=red] (-1,2) circle (0.05);
	\node[left] at (-1, 1.5) {$9_\ell$};
	\draw[red, fill=red] (-1,1.5) circle (0.05);
	\node[left] at (-1, 1) {$10_\ell$};
	\draw[red, fill=red] (-1,1) circle (0.05);
	\node[right] at (1,5.5) {$1_r$};
	\draw[fill=black] (1,5.5) circle (0.05);
	\node[right] at (1,5) {$2_r$};
	\draw[fill=black] (1,5) circle (0.05);
	\node[right] at (1,4.5) {$3_r$};
	\draw[fill=black] (1,4.5) circle (0.05);
	\node[right] at (1,4) {$4_r$};
	\draw[ggreen, fill=ggreen] (1,4) circle (0.05);
	\node[right] at (1,3.5) {$5_r$};
	\draw[fill=black] (1,3.5) circle (0.05);
	\node[right] at (1,3) {$6_r$};
	\draw[ggreen, fill=ggreen] (1,3) circle (0.05);
	\node[right] at (1,2.5) {$7_r$};
	\draw[red, fill=red] (1,2.5) circle (0.05);
	\node[right] at (1,2) {$8_r$};
	\draw[red, fill=red] (1,2) circle (0.05);
	\node[right] at (1,1.5) {$9_r$};
	\draw[red, fill=red] (1,1.5) circle (0.05);
	\node[right] at (1,1) {$10_r$};
	\draw[red, fill=red] (1,1) circle (0.05);
	\node[right] at (1,0.5) {$11_r$};
	\draw[red, fill=red] (1,0.5) circle (0.05);
	\node[right] at (1,0) {$12_r$};
	\draw[red, fill=red] (1,0) circle (0.05);
	\end{tikzpicture}
\end{align*}

As $\pi \in BNC_L(n,m)$, the block $V_\sigma$ of $\sigma$ that contain $1_\ell$ contains a $(2j)_r$ for some $k$.  Furthermore, due to the properties of the Kreweras complement, there exist $t,s\geq 1$, $1 = i_1 < i_2 < \cdots < i_t \leq n$, and $1 \leq j_1 < j_2 < \cdots < j_s \leq m$ such that
\[
V_\sigma = \{(2i_p-1)_\ell\}^t_{p=1} \cup \{(2j_q)_r\}^s_{q=1}.
\]
Note that $V_\sigma$ divides the remaining nodes into $t-1$ regions on the left, $s$ regions on the right (including the one with $1_r$), and one region at the bottom.

For each $1 \leq p \leq t$ let $d_p = i_{p+1} - i_{p}$ where $i_{t+1} = n+1$.  Note that $\sum^t_{p=1} d_p = n$. For each $1 \leq p \leq t-1$ (if any such $p$ exist) let $\sigma_{\ell, p}$ denote the non-crossing partition obtained by restricting $\sigma$ to $\{(2i_p)_\ell, (2i_p+1)_\ell, \ldots, (2i_{p+1}-2)_\ell\}$.  We omitted the case $p = t$ as $\{(2i_t)_\ell, (2i_t+1)_\ell, \ldots, (2n)_\ell\}$ belongs to the bottom region.   Notice for $1 \leq p \leq t-1$ that if $\sigma'_{\ell, p}$ is obtained from $\sigma_{\ell, p}$ by adding the singleton block $\{(2i_p-1)_\ell\}$, then $\sigma'_{\ell, p}$ is naturally an element of $NC'(d_p)$.  The below diagram demonstrates an example of this restriction.
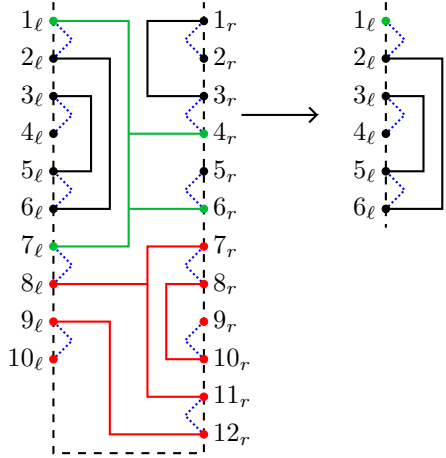
\begin{align*}
	\begin{tikzpicture}[baseline]
	\draw[thick] (1.5, 4.25) -- (2.5, 4.25) -- (2.4, 4.15);
	\draw[thick] (2.5,4.25) -- (2.4, 4.35);
	\draw[thick, dashed] (-1,5.75) -- (-1,-.25) -- (1,-.25) -- (1,5.75);
	\draw[thick, blue, densely dotted] (1, 5.5) -- (0.75, 5.25) -- (1, 5);
	\draw[thick, blue, densely dotted] (1, 4.5) -- (0.75, 4.25) -- (1, 4);
	\draw[thick, blue, densely dotted] (1, 3.5) -- (0.75, 3.25) -- (1, 3);
	\draw[thick, blue, densely dotted] (1, 2.5) -- (0.75, 2.25) -- (1, 2);
	\draw[thick, blue, densely dotted] (1, 1.5) -- (0.75, 1.25) -- (1, 1);
	\draw[thick, blue, densely dotted] (1, 0.5) -- (0.75, 0.25) -- (1, 0);
	\draw[thick, blue, densely dotted] (-1, 5.5) -- (-0.75, 5.25) -- (-1, 5);
	\draw[thick, blue, densely dotted] (-1, 4.5) -- (-0.75, 4.25) -- (-1, 4);
	\draw[thick, blue, densely dotted] (-1, 3.5) -- (-0.75, 3.25) -- (-1, 3);
	\draw[thick, blue, densely dotted] (-1, 2.5) -- (-0.75, 2.25) -- (-1, 2);
	\draw[thick, blue, densely dotted] (-1, 1.5) -- (-0.75, 1.25) -- (-1, 1);
	\draw[thick, ggreen] (-1,5.5) -- (0,5.5) -- (0,2.5) -- (-1,2.5);
	\draw[thick, ggreen] (1,4) -- (0,4);
	\draw[thick, ggreen] (1,3) -- (0,3);
	\draw[thick, black] (1,5.5) -- (0.25, 5.5) -- (0.25, 4.5) -- (1,4.5);
	\draw[thick, black] (-1,5) -- (-0.25, 5) -- (-0.25, 3) -- (-1,3);
	\draw[thick, black] (-1,4.5) -- (-0.5, 4.5) -- (-0.5, 3.5) -- (-1,3.5);
	\draw[thick, red] (1,2.5) -- (0.25,2.5) -- (0.25,.5) -- (1,.5);
	\draw[thick, red] (-1,2) -- (0.25,2);
	\draw[thick, red] (1,2) -- (0.5,2) -- (0.5,1) -- (1,1);
	\draw[thick, red] (-1,1.5) -- (-.25,1.5) -- (-0.25,0) -- (1,0);
	\node[left] at (-1, 5.5) {$1_\ell$};
	\draw[ggreen, fill=ggreen] (-1,5.5) circle (0.05);
	\node[left] at (-1, 5) {$2_\ell$};
	\draw[fill=black] (-1,5) circle (0.05);
	\node[left] at (-1, 4.5) {$3_\ell$};
	\draw[fill=black] (-1,4.5) circle (0.05);
	\node[left] at (-1, 4) {$4_\ell$};
	\draw[fill=black] (-1,4) circle (0.05);
	\node[left] at (-1, 3.5) {$5_\ell$};
	\draw[fill=black] (-1,3.5) circle (0.05);
	\node[left] at (-1, 3) {$6_\ell$};
	\draw[fill=black] (-1,3) circle (0.05);
	\node[left] at (-1, 2.5) {$7_\ell$};
	\draw[ggreen, fill=ggreen] (-1,2.5) circle (0.05);
	\node[left] at (-1, 2) {$8_\ell$};
	\draw[red, fill=red] (-1,2) circle (0.05);
	\node[left] at (-1, 1.5) {$9_\ell$};
	\draw[red, fill=red] (-1,1.5) circle (0.05);
	\node[left] at (-1, 1) {$10_\ell$};
	\draw[red, fill=red] (-1,1) circle (0.05);
	\node[right] at (1,5.5) {$1_r$};
	\draw[fill=black] (1,5.5) circle (0.05);
	\node[right] at (1,5) {$2_r$};
	\draw[fill=black] (1,5) circle (0.05);
	\node[right] at (1,4.5) {$3_r$};
	\draw[fill=black] (1,4.5) circle (0.05);
	\node[right] at (1,4) {$4_r$};
	\draw[ggreen, fill=ggreen] (1,4) circle (0.05);
	\node[right] at (1,3.5) {$5_r$};
	\draw[fill=black] (1,3.5) circle (0.05);
	\node[right] at (1,3) {$6_r$};
	\draw[ggreen, fill=ggreen] (1,3) circle (0.05);
	\node[right] at (1,2.5) {$7_r$};
	\draw[red, fill=red] (1,2.5) circle (0.05);
	\node[right] at (1,2) {$8_r$};
	\draw[red, fill=red] (1,2) circle (0.05);
	\node[right] at (1,1.5) {$9_r$};
	\draw[red, fill=red] (1,1.5) circle (0.05);
	\node[right] at (1,1) {$10_r$};
	\draw[red, fill=red] (1,1) circle (0.05);
	\node[right] at (1,0.5) {$11_r$};
	\draw[red, fill=red] (1,0.5) circle (0.05);
	\node[right] at (1,0) {$12_r$};
	\draw[red, fill=red] (1,0) circle (0.05);
	\end{tikzpicture}
	\quad
	\begin{tikzpicture}[baseline]
	\draw[thick, dashed] (-1,5.75) -- (-1,2.75);
	\draw[thick, blue, densely dotted] (-1, 5.5) -- (-0.75, 5.25) -- (-1, 5);
	\draw[thick, blue, densely dotted] (-1, 4.5) -- (-0.75, 4.25) -- (-1, 4);
	\draw[thick, blue, densely dotted] (-1, 3.5) -- (-0.75, 3.25) -- (-1, 3);
	\draw[thick, black] (-1,5) -- (-0.25, 5) -- (-0.25, 3) -- (-1,3);
	\draw[thick, black] (-1,4.5) -- (-0.5, 4.5) -- (-0.5, 3.5) -- (-1,3.5);
	\node[left] at (-1, 5.5) {$1_\ell$};
	\draw[ggreen, fill=ggreen] (-1,5.5) circle (0.05);
	\node[left] at (-1, 5) {$2_\ell$};
	\draw[fill=black] (-1,5) circle (0.05);
	\node[left] at (-1, 4.5) {$3_\ell$};
	\draw[fill=black] (-1,4.5) circle (0.05);
	\node[left] at (-1, 4) {$4_\ell$};
	\draw[fill=black] (-1,4) circle (0.05);
	\node[left] at (-1, 3.5) {$5_\ell$};
	\draw[fill=black] (-1,3.5) circle (0.05);
	\node[left] at (-1, 3) {$6_\ell$};
	\draw[fill=black] (-1,3) circle (0.05);
	\end{tikzpicture}
\end{align*}

Similarly, for each $1 \leq q \leq s+1$ let $e_q = j_q - j_{q-1}$, where $j_0 = 0$ and $j_{s+1} = m$. Note that $\sum^s_{q=1} j_q = m$.  For each $1 \leq q \leq s$ let $\sigma_{r, q}$ denote the non-crossing partition obtained by restricting $\sigma$ to $\{(2j_{q-1} + 1)_r, (2j_{q-1} + 2)_r, \ldots, (2j_{q}-1)_r\}$.  We omitted the case $q = s+1$ as $\{(2j_{s}+1)_r, (2j_{s}+2)_r, \ldots, (2m)_r\}$ belongs to the bottom region.  Notice for $1 \leq q \leq s$ that if $\sigma'_{r, q}$ is obtained from $\sigma_{r, q}$ by adding the singleton block $\{(2j_q)_r\}$, then $\sigma'_{r, q}$ is naturally an element of $NC'(e_q)$.

Finally let $\sigma_b$ denote the restriction of $\sigma$ to $\{(2i_t)_\ell, (2i_t+1)_\ell, \ldots, (2n)_\ell\} \cup \{(2j_{s}+1)_r, (2j_{s}+2)_r, \ldots, (2m)_r\}$.  Then $\sigma_b$ is easily seen to be an element of $BNC_{Lb}(2d_t-1, 2e_{s+1})$.

Now, by writing each of $\kappa_\pi(a_1, b_1)$ and $\kappa_{K(\pi)}(a_2, b_2)$ as a product of cumulants, we obtain that
\begin{align*}
\kappa_\pi(a_1, b_1)&\kappa_{K(\pi)}(a_2, b_2)z^n w^m \\
&= \kappa_{t, s}(a_1, b_1)   \prod^{t-1}_{p=1} f_1(0_{d_p},\sigma'_{\ell, p}) f_2 (0_{d_p},K(\sigma'_{\ell, p}))z^{d_p}    \prod^{s}_{q=1}  g_1(0_{e_q},\sigma'_{r, q}) g_2 (0_{e_q},K(\sigma'_{r, q})) w^{e_q} \\
& \quad \times \kappa_{\sigma_b}(\underbrace{a_2, a_1, a_2, a_1 \ldots, a_2, a_1, a_2}_{a_1 \text{ occurs }d_t-1 \text{ times}}, \underbrace{b_2, b_1, b_2, b_1, \ldots, b_2, b_1}_{b_1, b_2 \text{ occur }e_{s+1} \text{ times}} ) z^{d_t}w^{e_{s+1}}.
\end{align*}
Consequently, summing $\kappa_\pi(a_1, b_1)\kappa_{K(\pi)}(a_2, b_2)$ over all $\rho \in BNC_L(n,m)$ such that if $\tau$ is the element of $BNC(2n, 2m)$ corresponding to $(\rho, K(\rho))$ via Remark \ref{rem:Kreweras-BNC} then $V_\tau = V_\sigma$ and $\tau_b = \sigma_b$, we obtain 
\begin{align*}
&\kappa_{t, s}(a_1, b_1)   \prod^{t-1}_{p=1} (f_1\check{\ast} f_2)(0_{d_p}, 1_{d_p}) z^{d_p}    \prod^{s}_{q=1}  (g_1\check{\ast} g_2)(0_{e_q}, 1_{e_q}) w^{e_q} \\
& \quad \times \kappa_{\sigma_b}(\underbrace{a_2, a_1, a_2, a_1 \ldots, a_2, a_1, a_2}_{a_1 \text{ occurs }d_t-1 \text{ times}}, \underbrace{b_2, b_1, b_2, b_1, \ldots, b_2, b_1}_{b_1, b_2 \text{ occur }e_{s+1} \text{ times}} ) z^{d_t}w^{e_{s+1}}.
\end{align*}

Finally, if we sum over all possible $n,m\geq 1$ and all possible $\pi \in BNC_L(n,m)$, we will be summing over all possible $t,s \geq 1$, all $d_p\geq 1$ for all $1 \leq p \leq t$, all $e_q \geq 1$ for all $1 \leq p \leq s$, all $e_{s+1} \geq 0$, and all possible $\sigma_b \in BNC_{Lb}(2d_t-1, 2e_{s+1})$.  Therefore we obtain 
\[
\phi_L(z,w) = \frac{1}{\phi_{f_1  \check{\ast} f_2}(z)} K_{a_1, b_1}\left(\phi_{f_1  \check{\ast} f_2}(z), \phi_{g_1  \check{\ast} g_2}(w)\right) \psi_L(z,w)
\]
where we needed divide by $\phi_{f_1  \check{\ast} f_2}(z)$ due to the occurrence of $t-1$ instead of $t$ in the product.
\end{proof}

\begin{lem}
\label{lem:2}
Under the above notation and assumptions,
\[
\phi_R(z,w) = \frac{1}{\phi_{f_2  \check{\ast} f_1}(w)} K_{a_2, b_2}\left(\phi_{f_2  \check{\ast} f_1}(z), \phi_{g_2  \check{\ast} g_1}(w)\right) \psi_R(z,w)
\]
as holomorphic functions near $(0,0)$.
\end{lem}
\begin{proof}
The proof of this result can be obtained by applying a mirror to the proof of  Lemma \ref{lem:1}.
\end{proof}

To understand $\psi_L(z,w)$ and $\psi_R(z,w)$ we demonstrate the following.
\begin{lem}
\label{lem:3}
Under the above notation and assumptions,
\[
\psi_L(z,w) = \phi_{f_1  \check{\ast} f_2}(z) + \frac{z}{\phi_{f_2  \check{\ast} f_1}(z)  \phi_{g_2  \check{\ast} g_1}(w)} K_{a_2, b_2}\left(\phi_{f_2  \check{\ast} f_1}(z), \phi_{g_2  \check{\ast} g_1}(w)\right) \psi_R(z,w)
\]
as holomorphic functions near $(0,0)$.
\end{lem}
\begin{proof}
First observe that dividing by $\phi_{f_2  \check{\ast} f_1}(z)  \phi_{g_2  \check{\ast} g_1}(w)$ is valid (e.g. see the end of the proof for how this division term occurs).  Furthermore all terms involved in proof involve absolutely summable series and thus can be rearranged in any order we desired.

First observe for all $n \geq 0$ that $BNC_{Lb}(2n+1, 0) = NC'(n+1)$.  Therefore, summing over all $n \geq 1$ and elements of $BNC_{Lb}(2n+1, 0)$ in $\psi_L(z,w)$ directly yields $\phi_{f_1  \check{\ast} f_2}(z)$ as $\varphi(a_1) = 1$.  Thus we focus on the elements of $BNC_{Lb}(2n+1, 2m)$ with $m \geq 1$.

For now fix $n \geq 0$, $m \geq 1$, and a $\pi \in BNC_{Lb}(2n+1,2m)$.  Let $V_\pi$ denote the block of $\pi$ containing $1_\ell$.  Note that in order for $\pi \vee \sigma_L = 1_{2n+1, m}$, it must be the case that $V_\pi$ contains $1_r$.  Furthermore there must exists $t,s\geq 1$, $1 = i_1 < i_2 < \cdots < i_t \leq n$, and $1 = j_1 < j_2 < \cdots < j_s \leq m$ such that
\[
V_\pi = \{(2i_p-1)_\ell\}^t_{p=1} \cup \{(2j_q-1)_r\}^s_{q=1}.
\]
Note that $V_\pi$ divides the remaining nodes into $t-1$ regions on the left, $s-1$ regions on the right, and one region at the bottom.  Below is a diagram of such a $\pi$ with the same conventions as in Lemma \ref{lem:1} where $n = 5$ and $m = 4$.
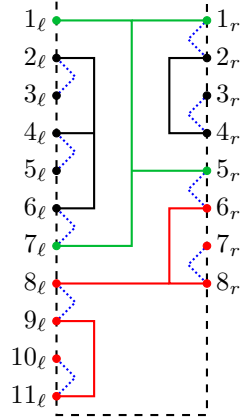
\begin{align*}
\begin{tikzpicture}[baseline]
	\draw[thick, dashed] (-1,6.25) -- (-1,.75) -- (1,.75) -- (1,6.25);
	\draw[thick, blue, densely dotted] (1, 5.5) -- (0.75, 5.75) -- (1, 6);
	\draw[thick, blue, densely dotted] (1, 4.5) -- (0.75, 4.75) -- (1, 5);
	\draw[thick, blue, densely dotted] (1, 3.5) -- (0.75, 3.75) -- (1, 4);
	\draw[thick, blue, densely dotted] (1, 2.5) -- (0.75, 2.75) -- (1, 3);
	\draw[thick, blue, densely dotted] (-1, 5.5) -- (-0.75, 5.25) -- (-1, 5);
	\draw[thick, blue, densely dotted] (-1, 4.5) -- (-0.75, 4.25) -- (-1, 4);
	\draw[thick, blue, densely dotted] (-1, 3.5) -- (-0.75, 3.25) -- (-1, 3);
	\draw[thick, blue, densely dotted] (-1, 2.5) -- (-0.75, 2.25) -- (-1, 2);
	\draw[thick, blue, densely dotted] (-1, 1.5) -- (-0.75, 1.25) -- (-1, 1);
	\draw[thick, ggreen] (-1,6) -- (0,6) -- (0,3) -- (-1,3);
	\draw[thick, ggreen] (1,4) -- (0,4);
	\draw[thick, ggreen] (1,6) -- (0,6);
	\draw[thick, black] (1,5.5) -- (0.5,5.5) -- (0.5,4.5) -- (1,4.5);
	\draw[thick, black] (-1,5.5) -- (-0.5,5.5) -- (-0.5,3.5) -- (-1,3.5);
	\draw[thick, black] (-0.5,4.5) -- (-1,4.5);
	\draw[thick, red] (1,2.5) -- (0.5,2.5) -- (0.5,3.5) -- (1,3.5);
	\draw[thick, red] (-1,2.5) -- (0.5,2.5);
	\draw[thick, red] (-1,2) -- (-0.5,2) -- (-0.5,1) -- (-1,1);
	\node[left] at (-1, 6) {$1_\ell$};
	\draw[ggreen, fill=ggreen] (-1,6) circle (0.05);	
	\node[left] at (-1, 5.5) {$2_\ell$};
	\draw[fill=black] (-1,5.5) circle (0.05);
	\node[left] at (-1, 5) {$3_\ell$};
	\draw[fill=black] (-1,5) circle (0.05);
	\node[left] at (-1, 4.5) {$4_\ell$};
	\draw[fill=black] (-1,4.5) circle (0.05);
	\node[left] at (-1, 4) {$5_\ell$};
	\draw[fill=black] (-1,4) circle (0.05);
	\node[left] at (-1, 3.5) {$6_\ell$};
	\draw[fill=black] (-1,3.5) circle (0.05);
	\node[left] at (-1, 3) {$7_\ell$};
	\draw[ggreen, fill=ggreen] (-1,3) circle (0.05);
	\node[left] at (-1, 2.5) {$8_\ell$};
	\draw[red, fill=red] (-1,2.5) circle (0.05);
	\node[left] at (-1, 2) {$9_\ell$};
	\draw[red, fill=red] (-1,2) circle (0.05);
	\node[left] at (-1, 1.5) {$10_\ell$};
	\draw[red, fill=red] (-1,1.5) circle (0.05);
	\node[left] at (-1, 1) {$11_\ell$};
	\draw[red, fill=red] (-1,1) circle (0.05);
	\node[right] at (1,6) {$1_r$};
	\draw[ggreen, fill=ggreen] (1,6) circle (0.05);	
	\node[right] at (1,5.5) {$2_r$};
	\draw[fill=black] (1,5.5) circle (0.05);
	\node[right] at (1,5) {$3_r$};
	\draw[fill=black] (1,5) circle (0.05);
	\node[right] at (1,4.5) {$4_r$};
	\draw[fill=black] (1,4.5) circle (0.05);
	\node[right] at (1,4) {$5_r$};
	\draw[ggreen, fill=ggreen] (1,4) circle (0.05);
	\node[right] at (1,3.5) {$6_r$};
	\draw[red, fill=red] (1,3.5) circle (0.05);
	\node[right] at (1,3) {$7_r$};
	\draw[red, fill=red] (1,3) circle (0.05);
	\node[right] at (1,2.5) {$8_r$};
	\draw[red, fill=red] (1,2.5) circle (0.05);
	\end{tikzpicture}
\end{align*}
For each $1 \leq p \leq t$ let $d_p = i_{p+1} - i_{p}$ where $i_{t+1} = n+1$.  Note that $\sum^t_{p=1} d_p = n$. For each $1 \leq p \leq t-1$ (if any such $p$ exist) let $\pi_{\ell, p}$ denote the non-crossing partition obtained by restricting $\pi$ to $\{(2i_p)_\ell, (2i_p+1)_\ell, \ldots, (2i_{p+1}-2)_\ell\}$.  We omitted the case $p = t$ as $\{(2i_t)_\ell, (2i_t+1)_\ell, \ldots, (2n+1)_\ell\}$ (which only exists if $i_t \neq n$ or, equivalently, $d_t > 1$) belongs to the bottom region.   Notice for $1 \leq p \leq t-1$ that if $\pi'_{\ell, p}$ is obtained from $\pi_{\ell, p}$ by adding the singleton block $\{(2i_p-1)_\ell\}$, then $\pi'_{\ell, p}$ is naturally an element of $NC'(d_p)$.

Similarly, for each $1 \leq q \leq s$ let $e_q = j_{q+1} - j_{q}$, where $j_{s+1} = m+1$. Note that $\sum^s_{q=1} j_q = m$.  For each $1 \leq q \leq s-1$ let $\pi_{r, q}$ denote the non-crossing partition obtained by restricting $\pi$ to $\{(2j_q)_r, (2j_q+1)_r, \ldots, (2j_{q+1}-2)_r\}$.  We omitted the case $q = s$ as $\{(2j_{s})_r, (2j_{s}+1)_\ell, \ldots, (2m)_r\}$ belongs to the bottom region.  Notice for $1 \leq q \leq s$ that if $\pi'_{r, q}$ is obtained from $\pi_{r, q}$ by adding the singleton block $\{(2j_q-1)_r\}$, then $\pi'_{r, q}$ is naturally an element of $NC'(e_q)$.

Finally let $\pi_b$ denote the restriction of $\pi$ to $\{(2i_p)_\ell, (2i_p+1)_\ell, \ldots, (2i_{p+1}-2)_\ell\} \cup \{(2j_{s})_r, (2j_{s}+1)_\ell, \ldots, (2m)_r\}$.  Then $\pi_b$ is easily seen to be an element of $BNC_{Rb}(2(d_t-1), 2e_{s}-1)$.

Now, by writing 
\[
\kappa_\pi(\underbrace{a_2, a_1, a_2, a_1 \ldots, a_2, a_1, a_2}_{a_1 \text{ occurs }n \text{ times}}, \underbrace{b_2, b_1, b_2, b_1, \ldots, b_2, b_1}_{b_1, b_2 \text{ occur }m \text{ times}} ) z^{n+1} w^m
\] 
as a product of cumulants, we obtain 
\begin{align*}
&z\kappa_{t, s}(a_2, b_2) \prod^{t-1}_{p=1} f_2(0_{d_p},\sigma'_{l, p}) f_1 (0_{d_p},K(\sigma'_{l, p}))z^{d_p}    \prod^{s-1}_{q=1}  g_2(0_{e_q},\sigma'_{r, q}) g_1 (0_{e_q},K(\sigma'_{r, q})) w^{e_q}\\
& \quad \times \kappa_{\sigma_b}(\underbrace{a_1, a_2, a_1, a_2 \ldots, a_1, a_2}_{a_1, a_2 \text{ occur }d_t-1 \text{ times}}, \underbrace{b_1, b_2, b_1, b_2, \ldots, b_1, b_2, b_1}_{b_2 \text{ occurs }e_{s}-1 \text{ times}} ) z^{d_t}w^{e_{s}}
\end{align*}
where the extra $z$ occurs as $\sum^t_{p=1} d_p = n$.  Consequently, by summing over all $\rho \in BNC_{Lb}(2n+1,2m)$ such that $V_\rho = V_\pi$ and $\rho_b = \pi_b$, we obtain 
\begin{align*}
&z\kappa_{t, s}(a_2, b_2)  \prod^{t-1}_{p=1} (f_2\check{\ast} f_1)(0_{d_p}, 1_{d_p}) z^{d_p}    \prod^{s-1}_{q=1}  (g_2\check{\ast} g_1)(0_{e_q}, 1_{e_q}) w^{e_q} \\
& \quad \times \kappa_{\sigma_b}(\underbrace{a_1, a_2, a_1, a_2 \ldots, a_1, a_2}_{a_1, a_2 \text{ occur }d_t-1 \text{ times}}, \underbrace{b_1, b_2, b_1, b_2, \ldots, b_1, b_2, b_1}_{b_2 \text{ occurs }e_{s}-1 \text{ times}} ) z^{d_t}w^{e_{s}}.
\end{align*}

Finally, if we sum over all possible $n\geq 0$, $m\geq 1$, and all possible $\pi \in BNC_{Lb}(2n+1,2m)$, we will be summing over all possible $t,s \geq 1$, all $d_p\geq 1$ for all $1 \leq p \leq t$, all $e_q \geq 1$ for all $1 \leq p \leq s$, and all possible $\pi_b \in BNC_{Rb}(2(d_t-1), 2e_{s}-1)$.  Therefore we obtain 
\[
\psi_L(z,w) = \phi_{f_1  \check{\ast} f_2}(z) + \frac{z}{\phi_{f_2  \check{\ast} f_1}(z)  \phi_{g_2  \check{\ast} g_1}(w)} K_{a_2, b_2}\left(\phi_{f_2  \check{\ast} f_1}(z), \phi_{g_2  \check{\ast} g_1}(w)\right) \psi_R(z,w)
\]
where we needed divide by $\phi_{f_2  \check{\ast} f_1}(z)  \phi_{g_2  \check{\ast} g_1}(w)$ due to the occurrence of $t-1$ and $s-1$ instead of $t$ and $s$ in the products.
\end{proof}

\begin{lem}
\label{lem:4}
Under the above notation and assumptions,
\[
\psi_R(z,w) = \phi_{g_2  \check{\ast} g_1}(w) + \frac{w}{\phi_{f_1  \check{\ast} f_2}(z)  \phi_{g_1  \check{\ast} g_2}(w)} K_{a_1, b_1}\left(\phi_{f_1  \check{\ast} f_2}(z), \phi_{g_1  \check{\ast} g_2}(w)\right) \psi_L(z,w)
\]
as holomorphic functions near $(0,0)$.
\end{lem}
\begin{proof}
The proof of this result can be obtained by applying a mirror to the proof of  Lemma \ref{lem:3}.
\end{proof}

The above equations are enough now complete the proof of Theorem \ref{thm:S-op-property} via some algebraic operations.  To simplify notation, let
\begin{align*}
\Phi_L(z,w) &:= \phi_L\left( \phi^\inv_{f_1  \ast f_2}(z), \phi^\inv_{g_1  \ast g_2}(w)\right)\\
\Phi_R(z,w) &:= \phi_R\left( \phi^\inv_{f_1  \ast f_2}(z), \phi^\inv_{g_1  \ast g_2}(w)\right) \\ 
\Psi_L(z,w) &:= \psi_L\left(  \phi^\inv_{f_1  \ast f_2}(z), \phi^\inv_{g_1  \ast g_2}(w)\right) \\
\Psi_R(z,w) &:= \psi_R\left(  \phi^\inv_{f_1  \ast f_2}(z), \phi^\inv_{g_1  \ast g_2}(w)\right).
\end{align*}

\begin{lem}
\label{lem:5}
Under the above notation and assumptions,
\[
\Psi_L(z,w) = \frac{\phi^\inv_{f_1}(z)\left(1 + \frac{1}{z} Q_2(z,w)\right)}{1 - \frac{1}{zw} Q_1(z,w) Q_2(z,w)} \qqand \Psi_R(z,w) = \frac{\phi^\inv_{g_2}(w)\left(1 + \frac{1}{w} Q_1(z,w)\right)}{1 - \frac{1}{zw} Q_1(z,w) Q_2(z,w)}
\]
as holomorphic functions near $(0,0)$.
\end{lem}
\begin{proof}
First observe both expressions are valid as holomorphic functions near $(0, 0)$ as the leading coefficient of both $Q_1(z,w)$ and $Q_2(z,w)$ is $zw$ so $1 - \frac{1}{zw} Q_1(z,w) Q_2(z,w)$ is invertible near $(0,0)$.  

Using Lemmata \ref{lem:3} and \ref{lem:4} along with equations (\ref{eq:pinch-convolution-technicality}, \ref{eq:convolution-with-inverse-series}), we obtain that
\begin{align*}
\Psi_L(z,w) &= \phi^\inv_{f_1}(z) + \frac{\phi^\inv_{f_1  \ast f_2}(z)}{\phi^\inv_{f_2}(z)  \phi^\inv_{g_2}(w)}Q_2(z,w) \Psi_R(z,w) \\
&= \phi^\inv_{f_1}(z) + \frac{\phi^\inv_{f_1}(z)}{z \phi^\inv_{g_2}(w)}   Q_2(z,w) \Psi_R(z,w)
\end{align*}
and
\begin{align*}
\Psi_R(z,w) &= \phi^\inv_{g_2}(w) + \frac{\phi^\inv_{g_1  \ast g_2}(w)}{\phi^\inv_{f_1}(z)  \phi^\inv_{g_1}(w)}Q_1(z,w) \Psi_L(z,w) \\
&= \phi^\inv_{g_2}(w) + \frac{\phi^\inv_{g_2}(w)}{w \phi^\inv_{f_1}(z)}   Q_1(z,w) \Psi_L(z,w).
\end{align*}
Substituting the first equation into the second equation yields
\begin{align*}
\Psi_R(z,w) &=\phi^\inv_{g_2}(w) + \frac{\phi^\inv_{g_2}(w)}{w \phi^\inv_{f_1}(z)}   Q_1(z,w) \left(\phi^\inv_{f_1}(z) + \frac{\phi^\inv_{f_1}(z)}{z \phi^\inv_{g_2}(w)}   Q_2(z,w) \Psi_R(z,w)\right) \\
&=\phi^\inv_{g_2}(w) + \frac{\phi^\inv_{g_2}(w)}{w}   Q_1(z,w) + \frac{1}{zw} Q_1(z,w)   Q_2(z,w) \Psi_R(z,w).
\end{align*}
Rearranging this equation and solving for $\Psi_R(z,w)$ yields the desired expression.  A similar substitution yields the expression for $\Psi_L(z,w)$.
\end{proof}

The proof of Lemma \ref{lem:K-expression} is now a simple computation.
\begin{proof}[Proof of Lemma \ref{lem:K-expression}]
First, by Lemmata \ref{lem:1} and \ref{lem:2} and equations (\ref{eq:pinch-convolution-technicality}, \ref{eq:Q-formula}) we see that
\[
\Phi_L(z,w) =\frac{1}{\phi^\inv_{f_1}(z)} Q_1(z,w) \Psi_L(z,w) \qqand \Phi_R(z,w) =\frac{1}{\phi^\inv_{g_2}(w)} Q_2(z,w) \Psi_R(z,w)
\]
By equations \eqref{eq:decompose-K}  and Lemma \ref{lem:5}, we see that
\begin{align*}
K_{a_1a_2,b_2b_1} \left( c^{\inv}_{a_1a_2}(z),  c^{\inv}_{b_2b_1}(w)\right) 
&= \Phi_L(z,w) + \Phi_R(z,w) \\
&= \frac{1}{\phi^\inv_{f_1}(z)} Q_1(z,w) \left(\frac{\phi^\inv_{f_1}(z)\left(1 + \frac{1}{z} Q_2(z,w)\right)}{1 - \frac{1}{zw} Q_1(z,w) Q_2(z,w)}\right) \\
& \qquad + \frac{1}{\phi^\inv_{g_2}(w)} Q_2(z,w) \left( \frac{\phi^\inv_{g_2}(w)\left(1 + \frac{1}{w} Q_1(z,w)\right)}{1 - \frac{1}{zw} Q_1(z,w) Q_2(z,w)}  \right) \\
&= \frac{Q_1(z,w) + Q_2(z,w) + \left(\frac{1}{z} + \frac{1}{w}\right) Q_1(z,w)Q_2(z,w)}{1 - \frac{1}{zw}Q_1(z,w)Q_2(z,w)}. \qedhere
\end{align*}
\end{proof}

More algebraic manipulation (and holomorphic extensions) will yield Theorem \ref{thm:S-op-property}.
\begin{proof}[Proof of Theorem \ref{thm:S-op-property}]
By multiply the numerator and denominator of the expression in Lemma \ref{lem:K-expression} by $\frac{1}{z} - \frac{1}{w}$ when $(z,w)$ are near zero and not equal yields
\begin{align*}
K_{a_1a_2,b_2b_1} \left( c^{\inv}_{a_1a_2}(z),  c^{\inv}_{b_2b_1}(w)\right) 
&= \frac{\left(\frac{1}{z} - \frac{1}{w}\right)\left(Q_1(z,w) + Q_2(z,w)\right) + \left(\frac{1}{z^2} - \frac{1}{w^2}\right) Q_1(z,w)Q_2(z,w)}{\left(\frac{1}{z} - \frac{1}{w}\right)\left(1 - \frac{1}{zw}Q_1(z,w)Q_2(z,w)\right)}.
\end{align*}
However, notice
\begin{align*}
&\frac{1}{z} \left(1 + \frac{1}{w} Q_1(z,w)\right) \left(1 + \frac{1}{w} Q_2(z,w)\right) - \frac{1}{w} \left(1 + \frac{1}{z} Q_1(z,w)\right) \left(1 + \frac{1}{z} Q_2(z,w)\right) \\
&= \left(\frac{1}{z} - \frac{1}{w}\right) \left(1 - \frac{1}{zw}Q_1(z,w)Q_2(z,w)\right).
\end{align*}
Hence, by multiply by the denominator in the expression for $K_{a_1a_2,b_2b_1} \left( c^{\inv}_{a_1a_2}(z),  c^{\inv}_{b_2b_1}(w)\right)$, moving all negative terms to the opposite side, adding one to both sides, and factoring, we obtain that
\begin{align*}
& \left(1 + \frac{1}{w} Q_1(z,w)\right)\left(1 + \frac{1}{w} Q_2(z,w)\right) \left(1 + \frac{1}{z}   K_{a_1a_2,b_2b_1} \left( c^{\inv}_{a_1a_2}(z),  c^{\inv}_{b_2b_1}(w)\right) \right)\\
&= \left(1 + \frac{1}{z} Q_1(z,w)\right)\left(1 + \frac{1}{z} Q_2(z,w)\right) \left(1 + \frac{1}{w}K_{a_1a_2,b_2b_1} \left( c^{\inv}_{a_1a_2}(z),  c^{\inv}_{b_2b_1}(w)\right) \right).
\end{align*}
By dividing both sides by 
\[
\left(1 + \frac{1}{w} Q_1(z,w)\right)\left(1 + \frac{1}{w} Q_2(z,w)\right)\left(1 + \frac{1}{w}K_{a_1a_2,b_2b_1} \left( c^{\inv}_{a_1a_2}(z),  c^{\inv}_{b_2b_1}(w)\right) \right)
\]
and applying holomorphic extensions then yields the result.
\end{proof}

\section{Other Settings}
\label{sec:OtherSettings}

In this final section, we discuss extensions of the opposite bi-free partial $S$-transform to both the operator-valued bi-free and the conditional bi-free settings.

\subsection{Operator-Valued Bi-Free $S$-Transforms}

In \cite{S2016-3} the author extended the bi-free partial $S$-transform to the operator-valued bi-free setting.  This generalization required a function of three arguments; one for left $B$-operators, one for right $B$-operators, and a central $B$-operator.  We briefly recall the results of \cite{S2016-3} here although we refer the reader to \cite{S2016-3} for full definitions.

\begin{defn}[\cite{S2016-3}*{Definition 8.1}]
Let $(\A, E, \varepsilon)$ be a Banach $B$-$B$-non-commutative probability space.  For each $b\in B$ let $L_b = \varepsilon(b \otimes 1)$ and $R_b = \varepsilon(1 \otimes b)$.  Let $X \in \A_\ell$ and $Y \in \A_r$ be such that $E(Y)$ and $E(X)$ are invertible.  For $b,c,d \in B$, define
\[
K_{X, Y}(b,c,d) := \sum_{n,m\geq 1} \kappa^B_{\chi_{n,m}}(  \underbrace{L_b X, \ldots, L_b X}_{n \text{ entries }}, \underbrace{R_d Y, \ldots, R_d Y}_{m-1 \text{ entries }}, R_d Y R_c)
\]
where $\kappa^B$ denotes the $B$-valued bi-free cumulant.

The \emph{operator-valued bi-free partial $S$-transform of $(X, Y)$} is the analytic function defined by
\begin{align*}
S_{X, Y}(b,c,d) &= c + b^{-1} \Upsilon_{X, Y}(b,c,d) + \Upsilon_{X, Y}(b,c,d) d^{-1} + b^{-1} \Upsilon_{X, Y}(b,c,d) d^{-1}
\end{align*}
where
\[
\Upsilon_{X, Y}(b,c,d) := K_{X, Y}\left(\Phi^\inv_{\ell, X}(b), c, \Phi^\inv_{r,Y}(d)\right) = K_{X, Y}\left(b S^\ell_X(b), c, S^r_Y(d)d \right)
\]
for any bounded collection of $c$ provided $b$ and $d$ sufficiently small.
\end{defn}

\begin{thm}[\cite{S2016-3}*{Theorem 8.3}]
\label{thm:operator-S}
Let $(\A, E, \varepsilon)$ be a Banach $B$-$B$-non-commutative probability space, let $(X_1, Y_1)$ and $(X_2, Y_2)$ be bi-free over $B$ with $E(X_k)$ and $E(Y_k)$ invertible for all $k$.  Then
\begin{align*}
S_{X_1 X_2, Y_1Y_2}& (b,c,d) \\ & = S^\ell_{X_2}(b)S_{X_1, Y_1}\left(S^\ell_{X_2}(b)^{-1}bS^\ell_{X_2}(b), \,\, S^\ell_{X_2}(b)^{-1}      S_{X_2, Y_2}(b,c,d) S^r_{Y_2}(d)^{-1}, \,\, S^r_{Y_2}(d) d S^r_{Y_2}(d)^{-1}\right) S^r_{Y_2}(d)
\end{align*}
for any bounded collection of $c$ provided $b$ and $d$ sufficiently small.
\end{thm}

Considering the above, it is natural to ask whether there is an operator-value version of the opposite bi-free partial $S$-transform.  However, this does not appear to be the case.

To prove Theorem \ref{thm:operator-S} it was possible to extend the proof of \cite{S2016-1}*{Theorem 4.5} because the bi-non-crossing diagrams that yielded $K_{X_2, Y_2}$ are encapsulated in those that produced $K_{X_1, Y_1}$ and hence $S_{X_2, Y_2}$ appeared as a term composed in $S_{X_1, Y_1}$. 

For the opposite bi-free partial $S$-transform as proved above, this encapsulation and composition do not take place.  In particular, the problem is with Lemmata \ref{lem:3}, \ref{lem:4}, and \ref{lem:5}.  Specifically the recursive nature of $\psi_L(z,w)$ and $\phi_R(z,w)$ is problematic in regards to extending the above proof to the operator-valued setting.  If one tries to extend Lemma \ref{lem:3} to the operator-valued setting, one obtains an expression for $\psi_L(z,w)$ in terms of a composition of $K_{a_2, b_2}$ with $\psi_R(z,w)$; that is, a term of the form $K_{a_2, b_2}(b, \psi_R(z,w), d)$.  Similarly, if one tries to extend Lemma \ref{lem:4} to the operator-valued setting, one obtains an expression for $\psi_R(z,w)$ in terms of a composition of $K_{a_1, b_1}$ with $\psi_L(z,w)$; that is, a term of the form $K_{a_1, b_1}(b, \psi_L(z,w), d)$.  If one attempts to solve for $\psi_L(z,w)$ as in Lemma \ref{lem:5}, the one obtains that $\psi_L(z,w)$ equals an expression involving $K_{a_2, b_2}(b', K_{a_1, b_1}(b, \psi_L(z,w), d), d')$.  Thus one cannot isolate and solve for $\psi_L(z,w)$ due to the properties of operator-valued bi-free cumulants and the inability to pull out the $c$-terms in $K_{X, Y}(b,c,d)$.  Even if one views these three variable functions as linear maps from $B$ to $B$ in the central variable, one only obtains that $\psi_L(z,w)$ is an inverse of a rather nasty linear map.

Thus, at least at this time, it appears the $S$-transform is more natural to the operator-valued setting then the opposite bi-free partial $S$-transform.

\subsection{Conditional Bi-Free $S$-Transforms}

In \cites{GS2016-1, GS2016-2} the study of conditional bi-free and operator-valued conditional bi-free probability was initiated. Using the conditional bi-free cumulants, analogues of the bi-free partial $R$-transforms were constructed in both the scalar and operator-valued conditional bi-free settings and used to discuss additive infinitely divisibility.  However, a conditional bi-free partial $S$-transform remains elusive.  Here we quickly record what occurs when one tries to repeat the combinatorial proof for the bi-free partial $S$-transform and the opposite bi-free partial $S$-transform in the conditional bi-free setting.

Given a two-state non-commutative probability space $(\A, \varphi, \psi)$ and $a,b \in \A$, let
\begin{align*}
c^c_a(z) &:= \sum_{n\geq 1} \kappa^c_n(a) z^n \\
K^c_{a,b}(z,w) &L= \sum_{n,m\geq 1} \kappa^c_{n,m}(a,b) z^n w^m
\end{align*}
where $\kappa^c_n(a)$ denotes the $n^{\mathrm{th}}$ conditional free cumulant of $a$ and $\kappa^c_{n,m}(a,b)$ denotes the $(n,m)$-th conditional bi-free cumulant of $(a,b)$.  

If $a_1$ and $a_2$ are conditionally free in with respect to $(\varphi, \psi)$ and have non-zero moments, then \cite{PW2011}*{Theorem 3.6} demonstrated that if
\[
S^c_a(z) = \frac{c^c_{a}\left(c^\inv_{a}(z)\right)}{c^\inv_{a}(z)}
\qquad
\text{then}
\qquad
S^c_{a_1a_2}(z) = S^c_{a_1}(z) S^c_{a_2}(z).  
\]
Thus, based on \cite{PW2011}, it is natural to consider
\[
K^c_{a,b}\left(c^\inv_{a}(z),c^\inv_{b}(w)\right)
\]
in the bi-free setting.

Suppose $(a_1, b_1)$ and $(a_2, b_2)$ are conditional bi-free pairs in $(\A, \varphi, \psi)$ with non-zero moments. If one repeats the arguments of \cite{S2016-1}*{Section 4} keeping track of which blocks are internal and which blocks are external, then one obtains by using results from \cite{PW2011} that 
\begin{align*}
K^c_{a_1a_2,b_1b_1} & \left(c^\inv_{a_1a_2}(z),c^\inv_{b_1b_2}(w)\right) \\
&= \frac{c^c_{a_2}\left(c^\inv_{a_2}(z)\right)}{z}\frac{c^c_{ab_2}\left(c^\inv_{b_2}(w)\right)}{w} K^c_{a_2, b_2} \left(c^\inv_{a_2}(z),c^\inv_{b_2}(w)\right) \\
& \qquad + \left[1 + \left(\frac{1}{z} + \frac{1}{w} + \frac{1}{zw}\right) K_{a_2, b_2} \left(c^\inv_{a_2}(z),c^\inv_{b_2}(w)\right)\right] K^c_{a_1, b_1} \left(c^\inv_{a_1}(z),c^\inv_{b_1}(w)\right).
\end{align*}
(As a sanity check, note that when $\varphi = \psi$ the conditional free and bi-free cumulants are equal to the free and bi-free cumulants respectively, and this reduces to the expression obtained in \cite{S2016-1}*{Section 4}).  However, even when combining this with the expression of $K_{a_1a_2,b_1b_2} \left(c^\inv_{a_1a_2}(z),c^\inv_{b_1b_2}(w)\right)$, it is not clear how to define $S^c_{a,b}(z,w)$ in order to obtain a product formula seeing as the expression for $K^c_{a_1a_2,b_1b_2} \left(c^\inv_{a_1a_2}(z),c^\inv_{b_1b_2}(w)\right)$ is not symmetric in $K^c_{a_1, b_1}$ and $K^c_{a_2, b_2}$ (i.e. if we had $S^c_{a_1a_2,b_1b_2}(z,w) = S^c_{a_1,b_1}(z,w)S^c_{a_2,b_2}(z,w)$ and $S^c_{a,b}(z,w)$ involved $K^c_{a,b} \left(c^\inv_{a}(z),c^\inv_{b}(w)\right)$, the right-hand side would be symmetric in $K^c_{a_1, b_1}$ and $K^c_{a_2, b_2}$).

If one repeats the arguments of Section \ref{sec:Proof} keeping track of which blocks are internal and which blocks are external, then one obtains by using results from \cite{PW2011} the conditional bi-free analogue of Lemma \ref{lem:K-expression} is
\begin{align*}
K^c_{a_1a_2,b_2b_1}& \left( c^{\inv}_{a_1a_2}(z),  c^{\inv}_{b_2b_1}(w)\right) \\
&= \frac{\frac{c_{b_2}^c\left(c^\inv_{b_2}(w)\right)}{w} \left(1+\frac{1}{z}Q_2(z,w)\right ) Q^c_1(z,w) + \frac{c_{a_1}^c\left(c^\inv_{a_1}(z)\right)}{z}\left(1+\frac{1}{w} Q_1(z,w)\right) Q^c_2(z,w) }{1 - \frac{1}{zw}Q_1(z,w)Q_2(z,w)}
\end{align*}
where
\begin{align*}
Q^c_j(z,w) = K^c_{a_j,b_j}\left( c^{\inv}_{a_j}(z),  c^{\inv}_{b_j}(w)\right). 
\end{align*}
(As a sanity check, note that when $\varphi = \psi$ the conditional free and bi-free cumulants are equal to the free and bi-free cumulants respectively, and this reduces to the expression obtained in Lemma \ref{lem:K-expression}).  However, even when combining this with the expression of $K_{a_1a_2,b_2b_1} \left(c^\inv_{a_1a_2}(z),c^\inv_{b_2b_1}(w)\right)$, it is not clear how to define $S^{c,\op}_{a,b}(z,w)$ in order to obtain a product formula seeing as the expression for $K^c_{a_1a_2,b_2b_1} \left(c^\inv_{a_1a_2}(z),c^\inv_{b_2b_1}(w)\right)$ is not symmetric in $K^c_{a_1, b_1}$ and $K^c_{a_2, b_2}$.

Thus, due to the lack of symmetry, conditional bi-free partial $S$-transforms appear difficult.  Even in the simplest case of bi-Boolean independence in \cite{GS2017}, no product formulae were obtainable.

\section*{Acknowledgements}

The author would like to thank Yinzheng Gu for informing him of the existence of \cite{HW2017}.


\begin{thebibliography}{13}


\bib{CNS2015-1}{article}{
    author={Charlesworth, I.},
    author={Nelson, B.},
    author={Skoufranis, P.},
    title={Combinatorics of Bi-Free Probability with Amalgamation},
	journal={Comm. Math. Phys.},
	volume={338},
	number={2},
    year={2015},
    pages={801-847}
}




\bib{CNS2015-2}{article}{
    author={Charlesworth, I.},
    author={Nelson, B.},
    author={Skoufranis, P.},
    title={On two-faced families of non-commutative random variables},
    journal={Canad. J. Math.},
    volume={67},
    number={6},
    year={2015},
    pages={1290-1325}
}

\bib{D2006}{article}{
    author={Dykema, K.},
    title={On the S-transform over a Banach algebra},
    journal={J. Funct. Anal.},
    number={1},
    volume={231},
    date={2006},
    pages={90-110}
}

\bib{GHM2015}{article}{
    author={Gu, Y.},
    author={Huang, H.-W.},
    author={Mingo, J.},
    title={An analogue of the L\'{e}vy-Hin\u{c}in formula for bi-free infinitely divisible distributions},
    journal={to appear in Indiana Univ. Math. J.},
    year={2015},
    pages={26 pages}
}


\bib{GS2016-1}{article}{
    author={Gu, Y.},
    author={Skoufranis, P.},
    title={Conditional Bi-Free Independence for Pairs of Algebras},
    eprint={arXiv:1609.07475},
    year={2016},
    pages={44}
}

\bib{GS2016-2}{article}{
    author={Gu, Y.},
    author={Skoufranis, P.},
    title={Conditional Bi-Free Independence with Amalgamation},
    eprint={arXiv:1609.07820},
    year={2016},
    pages={36}
}

\bib{GS2017}{article}{
    author={Gu, Y.},
    author={Skoufranis, P.},
    title={Bi-Boolean independence for pairs of algebras},
    eprint={arXiv:1703.03072},
    year={2017},
    pages={42}
}

\bib{H1997}{article}{
    author={Haagerup, U.},
    title={On Voiculescu's $R$- and $S$-transforms for free non-commuting variables},
    journal={in ``Free Probability Theory", D. V. Voiculescu, editor, Fields Institute Communications},
    volume={12},
    year={1997},
    pages={127-148}
}





\bib{HW2016}{article}{
    author={Huang, H.-W.},
    author={Wang, J.-C.},
    title={Analytic Aspects of Bi-Free Partial $R$-Transforms},
     journal={J. Funct. Anal.},
    volume={271},
    number={4},
    date={2016},
    pages={922-957}
}



\bib{HW2017}{article}{
    author={Huang, H.-W.},
    author={Wang, J.-C.},
    title={Harmonic Analysis for the Bi-Free Partial $S$-Transform},
    journal={preprint},
    date={2017},
    pages={28 pages}
}




\bib{NS1997}{article}{
    author={Nica, A.},
    author={Speicher, R.},
    title={A ``Fourier Transform" for Multiplicative Functions on Non-Crossing Partitions},
    journal={J. Algebraic Combin.},
    volume={6},
    year={1997},
    pages={141-160}
}






\bib{PW2011}{article}{
    author={Popa, M.},
    author={Wang, J. C.},
    title={On multiplicative conditionally free convolution},
    journal={Trans. Amer. Math. Soc.},
    volume={363},
    number={12},
    year={2011},
    pages={6309-6335}
}



\bib{S2016-1}{article}{
    author={Skoufranis, P.},
    title={A Combinatorial Approach to Voiculescu's Bi-Free Partial Transforms},
    journal={Pacific J. Math.},
    number={2},
    volume={283},
    year={2016},
    pages={419-447}
}




\bib{S2016-2}{article}{
    author={Skoufranis, P.},
    title={Independences and Partial $R$-Transforms in Bi-Free Probability},
    journal={Ann. Inst. Henri Poincar\'{e} Probab. Stat.},
    volume={52},
    number={3},
    pages={1437-1473},
    year={2016}
}



\bib{S2016-3}{article}{
    author={Skoufranis, P.},
    title={On Operator-Valued Bi-Free Distributions},
    journal={Adv. Math.},
    volume={303},
    year={2016},
    pages={638-715}
}




\bib{S1994}{article}{
    author={Speicher, R.},
    title={Multiplicative functions on the lattice of non-crossing partitions and free convolution},
    journal={Math. Ann.},
    volume={298},
    number={1},
    year={1994},
    pages={611-628}
}







\bib{V1986}{article}{
    author={Voiculescu, D.},
    title={Addition of certain non-commuting random variables},
    journal={J. Funct. Anal.},
    volume={66},
    number={3},
    year={1986},
    pages={323-346}
}





\bib{V1987}{article}{
    author={Voiculescu, D.},
    title={Multiplication of certain non-commuting random variables},
    journal={J. Operator Theory},
    volume={18},
    year={1987},
    pages={223-235}
}




\bib{V2014}{article}{
    author={Voiculescu, D.},
    title={Free probability for pairs of faces I},
    journal={Comm. Math. Phys.},
    year={2014},
    volume={332},
    pages={955-980}
}



\bib{V2016-1}{article}{
    author={Voiculescu, D. V.},
    title={Free Probability for Pairs of Faces II: 2-Variable Bi-Free Partial R-Transform and Systems with Rank $\leq 1$ Commutation},
    journal={Ann. Inst. Henri Poincar\'{e} Probab. Stat.},
    year={2016},
    volume={52},
    number={1},
    pages={1-15}
}


\bib{V2016-2}{article}{
    author={Voiculescu, D. V.},
    title={Free probability for pairs of faces III: 2-variable bi-Free partial $S$-transform and $T$-transforms},
    journal={J. Funct. Anal.},
    year={2016},
    volume={270},
    number={10},
    pages={3623-3638}
}



\end{thebibliography}
\end{document}